\documentclass[10pt,reqno]{amsart}
\usepackage[T1]{fontenc}
\usepackage[margin=1in]{geometry}
\usepackage{bbm}
\usepackage{amsfonts}
\usepackage{latexsym, amssymb, amsmath, amscd, amsthm, amsxtra}
\usepackage{mathtools}
\usepackage{enumerate}
\usepackage[all]{xy}
\usepackage{mathrsfs}
\usepackage{fancyhdr}
\usepackage{listings}
\usepackage{hyperref}
\usepackage{soul}
\usepackage{xcolor}
\usepackage{dsfont}
\usepackage{stmaryrd}
\usepackage{bm}
\usepackage{comment}
\usepackage{diagbox}
\usepackage{verbatim}
\usepackage{footmisc}
\usepackage{aliascnt}

\allowdisplaybreaks[4]
\hypersetup{colorlinks=true}
\usepackage{tikz}
\usetikzlibrary{decorations.pathmorphing, decorations.markings, decorations.pathreplacing}
\usetikzlibrary{arrows.meta}

\usepackage[rightcaption]{sidecap}

\usepackage{graphicx, color}

\numberwithin{equation}{section}
\numberwithin{figure}{section}

\makeatletter
\newcommand{\rmnum}[1]{\uppercase{{\expandafter{\romannumeral #1}}}}
\makeatother

\theoremstyle{plain} \newtheorem{theorem}{Theorem}[section]
\newtheorem*{theorem*}{Theorem}
\newaliascnt{lemma}{theorem}
\newtheorem{lemma}[lemma]{Lemma}
\aliascntresetthe{lemma}
\newtheorem*{lemma*}{Lemma}
\newaliascnt{corollary}{theorem}

\aliascntresetthe{corollary}
\newtheorem*{corollary*}{Corollary}
\newaliascnt{proposition}{theorem}

\aliascntresetthe{proposition}
\newtheorem*{proposition*}{Proposition}
\newaliascnt{definition}{theorem}
\newtheorem{definition}[definition]{Definition}
\aliascntresetthe{definition}
\newtheorem*{definition*}{Definition}
\newaliascnt{conjecture}{theorem}

\aliascntresetthe{conjecture}
\newtheorem*{conjecture*}{Conjecture}

\theoremstyle{definition}
\newaliascnt{example}{theorem}

\aliascntresetthe{example}
\newtheorem*{example*}{Example}
\newaliascnt{remark}{theorem}
\newtheorem{remark}[remark]{Remark}
\aliascntresetthe{remark}
\newtheorem*{remark*}{Remark}

\usepackage{cleveref}
\crefname{theorem}{Theorem}{Theorems}
\Crefname{theorem}{Theorem}{Theorems}
\crefname{lemma}{Lemma}{Lemmas}
\Crefname{lemma}{Lemma}{Lemmas}
\crefname{corollary}{Corollary}{Corollaries}
\Crefname{corollary}{Corollary}{Corollaries}
\crefname{proposition}{Proposition}{Propositions}
\Crefname{proposition}{Proposition}{Propositions}
\crefname{definition}{Definition}{Definitions}
\Crefname{definition}{Definition}{Definitions}
\crefname{conjecture}{Conjecture}{Conjectures}
\Crefname{conjecture}{Conjecture}{Conjectures}
\crefname{assumption}{Assumption}{Assumptions}
\Crefname{assumption}{Assumption}{Assumptions}
\crefname{example}{Example}{Examples}
\Crefname{example}{Example}{Examples}
\crefname{remark}{Remark}{Remarks}
\Crefname{remark}{Remark}{Remarks}
\crefname{claim}{Claim}{Claims}
\Crefname{claim}{Claim}{Claims}

 \newcommand{\ol}[1]{\overline{#1} \!\,} \newcommand{\wh}{\widehat}
\newcommand{\wt}{\widetilde}
\renewcommand{\txt}[1]{\text{\rm{#1}}}

\definecolor{darkred}{rgb}{0.9,0,0.3}
\definecolor{darkblue}{rgb}{0,0.3,0.9}

\usepackage{ifthen}
\def\comment#1{\ifthenelse{\isodd{\value{page}}}{\marginpar{\raggedright\scriptsize{\textcolor{darkred}{#1}}}}{\marginpar{\raggedleft\scriptsize{\textcolor{darkred}{#1}}}}}

\renewcommand{\P}{\mathbb{P}}
\newcommand{\E}{\mathbb{E}}
\newcommand{\R}{\mathbb{R}}
\newcommand{\C}{\mathbb{C}}

\newcommand{\Z}{\mathbb{Z}}

\newcommand{\cD}{{\mathcal{D}}}

\newcommand{\cN}{\mathcal{N}}

\newcommand{\sx}{\mathsf x}

\newcommand{\rd}{\mathrm{d}}

\newcommand{\ii}{\mathrm{i}}

 \renewcommand{\leq}{\le}
\renewcommand{\epsilon}{\varepsilon}

\newcommand{\ceil}[1]  {\lceil  {#1} \rceil}

\newcommand{\qq}[1]{[\![{#1}]\!]}

\newcommand{\p}[1]{({#1})}
\newcommand{\pb}[1]{\bigl({#1}\bigr)}
\newcommand{\pB}[1]{\Bigl({#1}\Bigr)}

\newcommand{\pa}[1]{\left({#1}\right)}

\newcommand{\q}[1]{[{#1}]}
\newcommand{\qb}[1]{\bigl[{#1}\bigr]}

\newcommand{\qbb}[1]{\biggl[{#1}\biggr]}

\newcommand{\qa}[1]{\left[{#1}\right]}

\newcommand{\h}[1]{\{{#1}\}}

\newcommand{\ha}[1]{\left\{{#1}\right\}}

\newcommand{\abs}[1]{\lvert #1 \rvert}

\newcommand{\absBB}[1]{\Biggl\lvert #1 \Biggr\rvert}
\newcommand{\absa}[1]{\left\lvert #1 \right\rvert}

\newcommand{\norm}[1]{\lVert #1 \rVert}
\newcommand{\normb}[1]{\bigl\lVert #1 \bigr\rVert}

\newcommand{\norma}[1]{\left\lVert #1 \right\rVert}

\newcommand{\bv}{{\bf{v}}}

\newcommand{\be}{\begin{equation}}
	\newcommand{\ee}{\end{equation}}

\newcommand{\oo}{\mathrm{o}}

\newcommand{\Zn}{\wt \Z_n}

\newcommand{\ZN}{\mathbb{Z}_N}

\newcommand{\Tr}{\mathrm{Tr}}

\newcommand{\blambda}{\boldsymbol{\lambda}}

\allowdisplaybreaks

\numberwithin{equation}{section}

\begin{document}
    \title{Poisson statistics of one-dimensional random band matrices}
	
	\author{Jiaqi Fan$^\star$}	
	\author{Guangyi Zou$^\dagger$}

	\thanks{\hspace{-9.3pt}$^\star$Qiuzhen College, Tsinghua University, Beijing, China, \href{mailto:fanjq24@mails.tsinghua.edu.cn}{fanjq24@mails.tsinghua.edu.cn}}
	\thanks{$^\dagger$Department of Mathematics, University of California, Irvine, \href{mailto:zouguangyi2001@gmail.com}{zouguangyi2001@gmail.com}}

	\begin{abstract}
		Consider an $N\times N$ random band matrix $H$ with bandwidth $W$ with centered real Gaussian entries. We prove that, under the (almost) sharp assumption $W^2\ll N/\log N$, the bulk local statistics of $H$ are asymptotically a Poisson point process. Combined with the bulk universality results from \cite{Band1D} and \cite{erdos2025zigzagstrategyrandomband}, this establishes the Poisson-RMT transition of the local statistics of random band matrices, and complements the results about the localization-delocalization transition of the bulk eigenvectors previously established in \cite{Band1D,erdos2025zigzagstrategyrandomband}, and \cite{Localization1_2}. The key technical inputs are: (1) the fractional moment estimates from \cite{Localization1_2}; (2) a novel estimate for the density of states, obtained from the local laws in \cite{Band1D} and \cite{erdos2025zigzagstrategyrandomband}.

	\end{abstract}
	\maketitle
	
	\tableofcontents

    \section{Introduction}\label{sec_introduction}

    The study of random band matrices dates back to the numerical work of Casati, Molinari and Izrailev \cite{scalingabndCGMLIF1990PRL} and the theoretical analysis of Fyodorov and Mirlin \cite{ScalingPropertyBandMatrixFYMA1991PRL}. These models have received much attention in both physics and mathematics because of their close connections to the Anderson model and quantum chaos. Roughly speaking, random band matrices are random matrices whose entries vanish outside a band around the diagonal. The bandwidth measures the range of interaction and provides a way to pass from models with short-range interactions to mean-field models such as Wigner matrices.

    For an $N\times N$ random band matrix in one dimension with bandwidth $W$, physical arguments predict that the bulk localization length is of order $W^2$, until it reaches the system size $N$. Thus, a transition is expected near $W\sim\sqrt N$: when $W\ll\sqrt N$, eigenvectors are localized and local eigenvalue statistics are Poisson; when $W\gg\sqrt N$, eigenvectors are delocalized and local statistics agree with those of the Gaussian orthogonal or unitary ensemble, according to the symmetry class. Numerical studies of band matrices and related quantum-chaotic systems gave early evidence for this picture \cite{ConJ-Ref2,scalingabndCGMLIF1990PRL,PhysRevLett.66.986}. Fyodorov and Mirlin \cite{ScalingPropertyBandMatrixFYMA1991PRL} derived the scaling parameter $W^2/N$ through a supersymmetric sigma-model calculation. The change in eigenvalue statistics is often called the Poisson--random matrix theory (Poisson--RMT) transition.

    The connection with the Anderson model is particularly useful. Introduced in the seminal work \cite{Anderson1958Absence} of Anderson, this model consists of a lattice Laplacian and an independent random potential. Scaling theory predicts localization in dimensions one and two, and a transition between localized and extended states in dimensions at least three \cite{PRL_Anderson}. Localization was proved in one dimension (see, e.g., \cite{JIMSPL1977FAIA,cmp/1103908590}), and at large disorder or near spectral edges in arbitrary dimensions through multiscale analysis and the fractional moment method \cite{frohlich1983absence,aizenman1993localization}. For local eigenvalue statistics, Molchanov proved a Poisson limit for a one-dimensional continuum model \cite{Molchanov1981LocalStructure}, and Minami established the corresponding result for the multidimensional discrete Anderson model under localization assumptions \cite{Minami1996Localfluctuation}. Further results were obtained in \cite{GerminetKloop2014Spectralstatistics}. In contrast, the existence of extended states at weak disorder for the standard Anderson model on $\mathbb Z^d$, $d\ge3$, remains open.

    We first review the rigorous results for one-dimensional random band matrices. On the delocalized side, partial results were obtained in \cite{erdHos2011quantum1,erdHos2011quantum,Average_fluc,delocal,HeMa2018,BaoErd2015,bourgade2017universality,bourgade2020random,bourgade2019random,Band1D_III,DY} over the past twenty years. More recently, Yau and Yin reached the sharp threshold $W\ge N^{1/2+\varepsilon}$ \cite{Band1D}, and Erd\H{o}s and Riabov extended the result to general entry distributions and variance profiles in both symmetry classes \cite{erdos2025zigzagstrategyrandomband}. A different extension to general profiles was developed by Fan, Yang and Yin through a block reduction method \cite{fan2025blockreductionmethodrandom}.

    On the localized side, the first quantitative eigenvector result was due to Schenker \cite{Sch2009}, who proved localization for Gaussian band matrices when $W\ll N^{1/8}$. Peled, Schenker, Shamis and Sodin improved this range to $W\ll N^{1/7}$ in \cite{Wegnerorbital2019RJMS}. Later, Chen and Smart \cite{Chen2022} and Cipolloni, Peled, Schenker and Shapiro \cite{Cipolloni2024} independently reached $W\ll N^{1/4}$. Drogin finally obtained the predicted $W^2$ decay scale for a broad class of random band matrices \cite{Localization1_2}.

    The eigenvalue problem in the localized regime is less complete. Brodie and Hislop proved convergence to a Poisson point process for fixed bandwidth, together with regularity and convergence of the density of states \cite{BrodieHislop2022}. For growing bandwidth, Hislop and Krishna obtained nontrivial Poisson subsequential limits for local eigenvalue counts in restricted bandwidth ranges, including $W=N^\alpha$ with $0<\alpha<1/7$ at almost every bulk energy \cite{HislopKrishna2022}. Supersymmetric methods also give precise results for special Gaussian models: Shcherbina and Shcherbina studied characteristic-polynomial correlations below the square-root scale \cite{ShcherbinaShcherbina2017} and proved GUE limits for the two-point correlation function above that scale, up to logarithmic factors \cite{ShcherbinaShcherbina2021}. At the spectral edge, Sodin \cite{Sodin2010Spectral} identified a different transition scale, $W\sim N^{5/6}$, and Liu and Zou extended the analysis to the critical regime and higher dimensions \cite{liu2025edgestatisticsrandomband,liu2025edgeuniversalityinhomogeneousrandom,liu2026edgeuniversalityinhomogeneousrandom}. Thus, although the bulk eigenvector transition is now understood at the predicted power scale, the full bulk Poisson--RMT transition has not been established under optimal conditions for random band matrices.

    In higher dimensions, Yang, Yau and Yin developed self-energy renormalization and $T$-expansions to prove delocalization in dimensions $d\geq 8$ \cite{yang2021delocalization,YYYTexpansion2022CMP}. Later, this was extended to dimensions $d\ge7$ by Xu, Yang, Yau and Yin \cite{xu2024bulk}. Recent work establishes delocalization and bulk universality for polynomially growing bandwidth in dimension two \cite{Band2D} and in every dimension $d\ge3$ \cite{dubova2025delocalizationnonmeanfieldrandommatrices}, extending the dynamical approach developed originally for one-dimensional RBMs in \cite{Band1D}.

    Random block Schr\"odinger operators provide another connection between the Anderson model and mean-field models. In the block Anderson model, the potential at each site is a random matrix and neighboring sites are joined by deterministic identity blocks. The block size measures the number of internal degrees of freedom. Strong-disorder localization and eigenvalue counting estimates for Gaussian orbital models were proved in \cite{Wegnerorbital2019RJMS}. Yang and Yin obtained delocalization results for a general class of block Schr\"odinger operators in high dimensions \cite{yang2025delocalizationgeneralclassrandom}, while Truong, Yang and Yin proved delocalization in dimension two and lower bounds on localization lengths in dimension one \cite{truong2025localizationlengthfinitevolumerandom}. For the mean-field block Anderson models, the complete transitions for both eigenvalues and eigenvectors throughout the whole spectrum are proved in \cite{stone2025random,fan2025localizationdelocalizationtransitionrandomblock}. More recently, Drogin proved sharp Lyapunov-exponent bounds for the one-dimensional GOE block Anderson model \cite{Drogin2026Furstenberg}.

    In this paper, we study the local eigenvalue statistics for the real Gaussian block band matrices defined in \Cref{def_considered_model} below. We prove convergence of the local eigenvalue process at every fixed bulk energy to a unit-intensity Poisson point process, after unfolding by the limiting density of states, whenever $W^2\log N/N\to0$; see \Cref{theorem_Poisson_process} below. As a byproduct, we also establish convergence and regularity properties of the density of states; see \Cref{lemma_rho_W_N} below. This gives the predicted localized eigenvalue statistics up to a logarithmic factor and, together with the results in the delocalized phase, establishes the conjectured transition in local statistics.

    \subsection{Main ideas}\label{sec_main_ideas}
    In the proof of the asymptotic Poisson point process (\Cref{theorem_Poisson_process}), we follow the classical framework developed in Minami's seminal work \cite{Minami1996Localfluctuation}. Roughly speaking, for any random matrix model $H$, this framework allows us to obtain the asymptotic Poisson point process as long as we have (1) \emph{fractional moment estimates} of the resolvent entries (\Cref{lemma_fractional_moments}), (2) sufficient \emph{control of the density of states} (\Cref{lemma_rho_W_N}), by the following two steps:
    \begin{enumerate}
        \item partitioning the base space $\ZN$ into smaller boxes, and proving the local statistics of the modified matrix $\wh H$, obtained by removing the hopping terms between these boxes, are asymptotically a Poisson point process;
        \item applying resolvent comparison to show that the local statistics of $H$ and $\wh H$ are asymptotically the same.
    \end{enumerate}
    Here, the first step uses the estimates of the density of states together with the standard Wegner-Minami estimates (\Cref{lemma_Wegner_Minami_estimates}), whereas the second relies mainly on the fractional moment estimates. While the Wegner-Minami and fractional moment estimates have been established in \cite{Wegnerorbital2019RJMS} and \cite{Localization1_2}, the primary novelty of this work is that we introduce the \emph{local law} of random band matrices in the delocalized phase from \cite{Band1D,erdos2025zigzagstrategyrandomband} to control the density of states.

    \medskip

    \paragraph{\bf Organization of the remaining text.} The remainder of the paper is organized as follows. In \Cref{sec_model_and_main_results}, we define the model and state the main result. \Cref{sec_preliminaries} collects the technical inputs for the proof of the main result. More specifically, we state and prove some standard estimates in \Cref{sec_a_prioro_estimates}, and establish the control of the density of states in \Cref{sec_density_of_states}. The proof in \Cref{sec_density_of_states} relies on two comparison lemmas, whose proofs are similar and hence postponed to \Cref{sec_proof_of_lemma_comparison}. With these preparations, we present the proof of the main result in \Cref{sec_proof_of_main_result}. Finally, we collect some open problems related to this work in \Cref{sec_open_problems}.

    \subsection*{Acknowledgement}
     We would like to thank Fan Yang for helpful comments. Large language models were used in a role analogous to that of a collaborator. J. F. is supported by NSFC (No.~12526201). G. Z. is supported by NSF Grant DMS 2451011 and U.S. Air Force Grant FA9550-25-1-0294.

    \subsection{Model and main results}\label{sec_model_and_main_results}
    In this work, as in \cite{Localization1_2} and \cite{Band1D}, we consider one-dimensional random band matrices with a block variance profile, defined as follows.
    \begin{definition}[Block random band matrices]\label{def_considered_model}
        Given any $W\in\Z_+$ and $N=nW$ with $n\in\Z_+$, we define the block random band matrix $H\equiv H_{W,N}$ with bandwidth $W$ on the circular lattice $\ZN:=\Z/N\Z$ as the real symmetric random matrix with independent centered Gaussian upper-triangular entries satisfying
        \begin{equation}
            \E H_{xy}^2=\frac{\mathbf{1}\pa{|\qa{x}-\qa{y}|\leq 1}}{3W},\qquad \forall x,y\in\ZN.
        \end{equation}
        Here, $\qa{x}$ and $\qa{y}$ are respectively the equivalence classes of $x$ and $y$ in the block lattice $\Zn:=\ZN/\sim_{\txt{block}}$, defined as the quotient lattice by the equivalence relation $\sim_{\txt{block}}$, and $|\qa{x}-\qa{y}|$ denotes the graph distance on $\Zn$. Specifically, given any $x,y\in\ZN$, we say $x\sim_{\txt{block}} y$, if and only if, for some $i,j\in\ha{0,1,\ldots,W-1}$ and $\ell\in\ha{0,1,\ldots,n-1}$,
        \begin{equation}
              x=\ell W+i \qquad \txt{and} \qquad y=\ell W+j.
        \end{equation}
        With a slight abuse of notation, we will also regard $\qa{x}\in\Zn$ as the block containing $x$. 
        For ease of presentation, we assume $n\geq 100$ to avoid the discussion of the boundary case, whenever considering the model defined here.\footnote{Clearly, under the assumption \eqref{eq_localization_assumption} below, we have $n\to \infty$. Moreover, when $n< 100$, the matrix $H$ is simply a mean-field Wigner-type matrix; see, e.g., \cite{ajanki2017universality} for the universality results.} 
        Moreover, given a block random band matrix $H$ defined above, we denote by $H^{\txt{o}}\equiv H^{\txt{o}}_{W,N}$ the matrix obtained from $H$ by removing the lower-left and upper-right corner blocks, $H_{\qa{0}\qa{N-1}}$ and $H_{\qa{N-1}\qa{0}}$. The matrix $H^{\txt{o}}$ will be referred to as the \emph{open-boundary form} of $H$, while the original matrix $H$ will be referred to as the \emph{periodic form}.
    \end{definition}

In block form, $H$ and its open-boundary form $H^{\mathrm{o}}$ are given by
\[
\begin{aligned}
H&=
\begin{pmatrix}
A_0 & B_0 & 0 & \cdots & B_{n-1}^{\mathsf T} \\
B_0^{\mathsf T} & A_1 & B_1 & \ddots & 0 \\
0 & B_1^{\mathsf T} & A_2 & \ddots & \vdots \\
\vdots & \ddots & \ddots & \ddots & B_{n-2} \\
B_{n-1} & 0 & \cdots & B_{n-2}^{\mathsf T} & A_{n-1}
\end{pmatrix},
\qquad H^{\mathrm{o}}=
\begin{pmatrix}
A_0 & B_0 & 0 & \cdots & 0 \\
B_0^{\mathsf T} & A_1 & B_1 & \ddots & \vdots \\
0 & B_1^{\mathsf T} & A_2 & \ddots & 0 \\
\vdots & \ddots & \ddots & \ddots & B_{n-2} \\
0 & \cdots & 0 & B_{n-2}^{\mathsf T} & A_{n-1}
\end{pmatrix}.
\end{aligned}
\]
Here all displayed blocks are $W\times W$, and $A_j=A_j^{\mathsf T}$.
The blocks $A_0,\ldots,A_{n-1},B_0,\ldots,B_{n-1}$ are independent.
The upper-triangular entries of each $A_j$ and all entries of each $B_j$
are independent Gaussian random variables with second moment $1/(3W)$.
Thus $H$ is a cyclic block tridiagonal matrix, and $H^{\mathrm{o}}$ is
obtained by setting the two corner blocks $B_{n-1}^{\mathsf T}$ and
$B_{n-1}$ to zero.

    \begin{remark}
        The block structures of the lattice and the matrix model are assumed for technical convenience. In fact, the method in this work is robust enough and extends to much more general models, e.g., the one-dimensional random band matrices with cutoff profile or complex Hermitian symmetry, as well as their open-boundary forms. We will address these extensions in the future versions of this work.
    \end{remark}

    Given any random band matrix $H$ defined in \Cref{def_considered_model} on the lattice $\ZN$, we define the \emph{resolvent} (also referred to as \emph{Green's function}) of $H$ at spectral parameter $z\in\C$ as
    \begin{equation}
        G(z)\equiv G_{W,N}(z):=\p{H_{W,N}-z}^{-1}, \qquad G^{\txt{o}}(z)\equiv G^{\txt{o}}_{W,N}(z):=\p{H^{\txt{o}}_{W,N}-z}^{-1}.
    \end{equation}
    Clearly, for every fixed $z\in\C$, the resolvent at $z$ is well-defined almost surely. Therefore, it makes sense to regard the resolvents as random matrices. Moreover, the Wegner estimate in \eqref{eq_Wegner_Minami_estimate} below shows the expectation of the empirical spectral measure is absolutely continuous with respect to the Lebesgue measure, and a smooth version of the Radon-Nikodym derivative, denoted as $\rho_{W,N}$, exists; see \Cref{lemma_rho_W_N_regularity} below. Specifically, $\rho_{W,N}$ is uniquely determined (up to Lebesgue null sets) by the following identity:
    \begin{equation}\label{def_rho_W_N}
        \frac{1}{N}\E[\Tr\,f(H_{W,N})]=\int_{\R}f(x)\rho_{W,N}(x)\, \rd x,\qquad \forall f\in C_{c}(\R).
    \end{equation}
    Clearly, by the definition of the Radon-Nikodym derivative, \eqref{def_rho_W_N} remains valid for every bounded Lebesgue measurable function $f$. 
    Furthermore, as illustrated in \Cref{lemma_rho_W_N} below, for every fixed $W\in\Z_+$, the expected spectral density $\rho_{W,N}$ converges locally uniformly to a continuous strictly positive probability density $\rho_{W}$ as $N\to \infty$. If we further let $W\to \infty$, the estimate \eqref{eq_rho_W_to_rho_sc} below shows that the limit density $\rho_{W}$ converges to the following Wigner semicircle law:
    \begin{equation}
        \rho_{\txt{sc}}(x):=\frac{1}{2\pi}\sqrt{(4-x^2)_+},
    \end{equation}
    whose Stieltjes transform will be referred to as $m_{\txt{sc}}$, defined by
    \begin{equation}
        m_{\txt{sc}}(z):=\int_{\R}\frac{\rho_{\txt{sc}}(u)}{u-z}\, \rd u=\frac{-z+\sqrt{z^2-4}}{2},\qquad \forall z\in\C\setminus[-2,2].
    \end{equation}
    Here, the branch is chosen so that $m_{\txt{sc}}(z)\sim -1/z$ at infinity. 
    Compared to the open form $H^{\mathrm{o}}$, the symmetries of the periodic form $H$ imply that, for every bounded measurable function $f$ and every $x\in\ZN$,
    \begin{equation}\label{eq_xx_to_1_N_Tr}
        \E f(H)_{xx}=\frac{1}{N}\E \Tr\, f(H).
    \end{equation}

    To state the main result of this work, given any fixed energy $E\in\R$, we denote the eigenvalues of $H_{W,N}$ as $\ha{\lambda_{k}}_{k=1}^{N}\equiv\ha{\lambda_{k}(H_{W,N})}_{k=1}^{N}$, and arrange them in the non-decreasing order $\lambda_1\leq \cdots \leq \lambda_{N}$.
    \begin{theorem}[Asymptotic Poisson point process]\label{theorem_Poisson_process}
        Suppose $H\equiv H_{W,N}$ is a random band matrix with bandwidth $W\geq 1$ and $N=nW$, defined in \Cref{def_considered_model}. Under the assumption
        \begin{equation}\label{eq_localization_assumption}
            W^2\log N/N\to 0,
        \end{equation}
        for every fixed bulk energy $E\in(-2,2)$, as $N\to \infty$, the following random measure converges in distribution to the Poisson point process $\txt{PPP}\pa{\rd \lambda}$ in the space of locally finite measures equipped with the vague topology:
        \begin{equation}\label{def_xi_E_W_N}
            \xi_{E}\equiv\xi_{E,W,N}:=\sum_{k=1}^{N}\delta_{N\rho_W(E)(\lambda_k-E)}.
        \end{equation}
        Here, $\txt{PPP}(\mu)$ denotes the Poisson point process on $\R$ with intensity measure $\mu$, and $\rd \lambda$ is the Lebesgue measure on $\R$. Equivalently, we have, for every $f\in C_{c}(\R)$ with $f\geq 0$,
        \begin{equation}\label{eq_xi_to_PPP_d_lambda}
            \lim_{N\to\infty} \E\exp\pa{-\int_{\R}f\, \rd \xi_{E}}=\exp\pa{-\int_{\R}\qa{1-\exp\pa{-f(\lambda)}}\, \rd \lambda}.
        \end{equation}
        If we further assume both $W,N\to \infty$, the convergence becomes
        \begin{equation}\label{eq_convergence_to_rho_sc_E_d_lambda}
            \sum_{k=1}^{N}\delta_{N(\lambda_k-E)}\implies \txt{PPP}\pa{\rho_{\txt{sc}}(E)\,\rd \lambda}.
        \end{equation}
        Here, the notation $\implies$ denotes convergence in distribution with respect to the vague topology. Moreover, the convergence also holds in the sense of $k$-point correlation functions of $\xi_{E}$, i.e., for every fixed $k\in\Z_+$ and every $F\in C_{c}(\R^k)$, we have
        \begin{equation}
            \lim_{N\to \infty}\frac{1}{[N\rho_W(E)]^k} \int_{\R^k} F(\blambda)\cdot\mathbf{r}^{(k)}_{N}\pa{E+\frac{\lambda_1}{N\rho_W(E)},\cdots,E+\frac{\lambda_k}{N\rho_W(E)}}\, \rd \blambda=\int_{\R^k} F(\blambda)\, \rd \blambda.
        \end{equation}
        Here, $\blambda$ denotes $\blambda=(\lambda_1,\ldots,\lambda_k)$, and $\mathbf{r}_{N}^{(k)}$ is the $k$-point correlation function (for existence, see, e.g., \cite[Theorem 3.10]{EvansGariepy2015}), defined uniquely up to Lebesgue null sets by the identity
        \begin{equation}
            \sum_{1\leq i_1,\ldots,i_{k}\leq N}^{\neq}\E F(\lambda_{i_1},\ldots,\lambda_{i_k})=\int_{\R^k}F(\blambda)\cdot\mathbf{r}_{N}^{(k)}(\blambda)\, \rd \blambda,\qquad \forall F\in C_{c}(\R^k),
\end{equation}
        where the sum on the left-hand side is taken over all pairwise distinct indices $i_1,\ldots,i_k$.

\end{theorem}

    \section{Preliminaries}\label{sec_preliminaries}

    \subsection{A priori estimates}\label{sec_a_prioro_estimates}
    As discussed in \Cref{sec_main_ideas}, the proof of \Cref{theorem_Poisson_process} combines the fractional moment estimates from \cite{Localization1_2} with control of the density of states obtained by the recent local law from \cite{Band1D} and \cite{erdos2025zigzagstrategyrandomband}, then employs the framework in which localization of eigenvectors and control of the density of states lead to Poisson local statistics. For ease of reference, in this section, we collect below the a priori estimates that will be used in subsequent sections. The estimates in this subsection do not require assumption \eqref{eq_localization_assumption}. We begin by recalling the Wegner and Minami estimates from \cite[Theorem 2 and 3]{Wegnerorbital2019RJMS}. For any symmetric matrix $H$ and measurable subset $I\subseteq \R$, define the eigenvalue counting function by
    \begin{equation}
        \cN_I(H):=\Tr\,\mathbf{1}_I(H).
    \end{equation}
    \begin{lemma}[Wegner and Minami estimates]\label{lemma_Wegner_Minami_estimates}Suppose $H\equiv H_{W,N}$ is a random band matrix with bandwidth $W\geq 1$ and $N=nW$, defined in \Cref{def_considered_model}. There exists a large constant $C\geq 1$, such that the following estimates hold uniformly over $N,W$, intervals $I\subseteq \R$, and integers $k\geq 1$:
        \begin{equation}\label{eq_Wegner_Minami_estimate}
            \E \cN_I(H)(\cN_I(H)-1)\cdots (\cN_I(H)-k+1)\leq (CN|I|)^k.
        \end{equation}
        Moreover, for any deterministic real symmetric matrix $A\in\R^{\ZN\times\ZN}$, with the same constant $C$, the bound \eqref{eq_Wegner_Minami_estimate} remains valid with $H$ replaced by any of $\ha{H^{\mathrm{o}},H+A,H^{\mathrm{o}}+A}$.
    \end{lemma}
    
    \begin{proof}
        By the definition of $H$ (recall \Cref{def_considered_model}), for any deterministic real symmetric matrix $A\in\R^{\ZN\times\ZN}$, we can write
        \begin{equation}
            H+A=\mathrm{diag}\, (V_1,\ldots,V_n)/100+H_0,
        \end{equation}
        where $\ha{V_1,\ldots,V_n}$ are independent $W\times W$ GOE matrices and $H_0$ is a random matrix independent of them. Then, the Wegner-Minami estimate \eqref{eq_Wegner_Minami_estimate} follows from Theorems 2 and 3 of \cite{Wegnerorbital2019RJMS}. The bound for the open-boundary form $H^{\mathrm{o}}$ can be proved by replacing each $H$ in the above argument with $H^{\mathrm{o}}$. This completes the proof of \Cref{lemma_Wegner_Minami_estimates}.
\end{proof}

    Next, we provide some basic regularity estimates on the expected spectral density $\rho_{W,N}$ defined in \eqref{def_rho_W_N}.

    \begin{lemma}[Regularity of the expected spectral density]\label{lemma_rho_W_N_regularity}
        Suppose $H\equiv H_{W,N}$ is a random band matrix with bandwidth $W\geq 1$ and $N=nW$, defined in \Cref{def_considered_model}. There exists a unique smooth function $\rho_{W,N}$ satisfying \eqref{def_rho_W_N}. Moreover, for some large constant $C\geq 1$, we have
        \begin{equation}\label{eq_regularity_rho_W_N}
            \norm{\rho_{W,N}}_{\infty}\leq C,\qquad \norm{\rho_{W,N}'}_{\infty}\leq C\p{NW}^{1/2}.
        \end{equation}
        Consequently, there exists a large constant $C\geq 1$ such that, for the Poisson kernel $P_{\eta}(\lambda):=\pi^{-1}\eta/(\lambda^2+\eta^2)$ with any $\eta\in(0,1/2)$, the following approximation holds:
        \begin{equation}\label{eq_rho_W_N_Poisson_approximation}
            \norm{\rho_{W,N}-P_{\eta}*\rho_{W,N}}_{\infty}\leq C(NW)^{1/2}\eta\log\eta^{-1}+C\eta.
        \end{equation}
\end{lemma}

    \begin{proof}
        By the Wegner estimate in \eqref{eq_Wegner_Minami_estimate} and the Radon-Nikodym theorem, there exists a bounded measurable function $\rho_{W,N}$ satisfying \eqref{def_rho_W_N} for any bounded measurable function $f$. To obtain a smooth version, let $H^{(1)},H^{(2)}$ be two independent copies of $H$. Since the entries are centered Gaussian random variables, we have
        \begin{equation}
            H\stackrel{\mathrm{d}}{=}\frac{1}{\sqrt{2}}H^{(1)}+\frac{1}{\sqrt{2}}H^{(2)}.
        \end{equation}
        Let $h:=N^{-1}\Tr\, H^{(2)}$, and write
        \begin{equation}
            H\stackrel{\mathrm{d}}{=}\frac{1}{\sqrt{2}}H^{(1)}+\frac{1}{\sqrt{2}}H^{(2)}=\qa{\frac{1}{\sqrt{2}}H^{(1)}+\frac{1}{\sqrt{2}}(H^{(2)}-h)}+\frac{h}{\sqrt{2}}=:\wh H+\frac{h}{\sqrt{2}}.
        \end{equation}
        By a direct computation of the covariance, for any $a,b\in\ZN$, we have
        \begin{equation}
            \mathrm{Cov}(\wh H_{ab},h)=\mathrm{Cov}\pa{\frac{1}{\sqrt{2}}H^{(1)}_{ab}+\frac{1}{\sqrt{2}}(H^{(2)}_{ab}-h\delta_{ab}),h}=\frac{\delta_{ab}}{\sqrt{2}}\E\pa{\frac{1}{N}\E H_{aa}^2-\frac{1}{N}\E H_{aa}^2}=0,
        \end{equation}
        which shows that the Gaussian matrix $\wh H$ is independent of $h$. Moreover, by the same argument as that at the beginning of this proof, there exists a bounded measurable probability density $\wh \rho\equiv \wh \rho_{W,N}$, such that \eqref{def_rho_W_N} holds with $(H,\rho)$ there replaced by $(\wh H,\wh \rho)$. Let $p_{h}(\lambda):=\sqrt{3NW/\pi}\exp\pa{-3NW\lambda^2}$ denote the probability density of $h/\sqrt{2}$. Then, for any bounded measurable function $f$, we have
        \begin{equation}
            \begin{aligned}
                &\int_{\R}f(\lambda)\wh \rho*p_{h}(\lambda)\, \rd \lambda=\int_{\R}f(\lambda)\int_{\R}\wh \rho(\lambda-u)p_{h}(u)\, \rd u\, \rd \lambda\\
                &\qquad\qquad=\frac{1}{N}\int_{\R}p_h(u)\E\Tr\,f(\wh H+u)\, \rd u=\frac{1}{N}\E\Tr\, f\pa{\wh H+\frac{h}{\sqrt{2}}}=\frac{1}{N}\E\Tr\, f(H).
            \end{aligned}
        \end{equation}
        Therefore, $\rho_{W,N}=\wh \rho_{W,N}*p_{h}$ up to a Lebesgue null set. This yields a bounded smooth version of $\rho_{W,N}$. Let $C\geq 1$ be a constant satisfying $\norm{\wh \rho_{W,N}}_{\infty}\leq C$. Then, we have
        \begin{equation}
            \norm{\rho_{W,N}'}_{\infty}\leq \norm{\wh \rho_{W,N}}_{\infty}\norm{p_h'}_{1}\leq 2C\pa{3NW/\pi}^{1/2},
        \end{equation}
        which yields the second bound in \eqref{eq_regularity_rho_W_N}. With \eqref{eq_regularity_rho_W_N}, the Poisson approximation \eqref{eq_rho_W_N_Poisson_approximation} follows from standard calculus. For example, uniformly in $E\in\R$, we have
        \begin{equation}
            \begin{aligned}
                &\absa{\rho_{W,N}(E)-P_{\eta}*\rho_{W,N}(E)}=\absa{\int_{\R}P_{\eta}(\lambda)(\rho_{W,N}(E-\lambda)-\rho_{W,N}(E))\, \rd \lambda}\\
                &\leq C(NW)^{1/2}\int_{|\lambda|\leq 1}|\lambda|P_{\eta}(\lambda)\, \rd \lambda+2C\int_{|\lambda|>1}P_{\eta}(\lambda)\, \rd \lambda\leq C'(NW)^{1/2}\eta\log\eta^{-1}+C'\eta.
            \end{aligned}
        \end{equation}
        This completes the proof of \Cref{lemma_rho_W_N_regularity}.
\end{proof}

    By the definition of the resolvent, for every $E\in \R$ and $\eta>0$, we have the identity
    \begin{equation}\label{eq_Poisson_rho_E_Im_Tr_G}
        P_{\eta}*\rho_{W,N}(E)=\frac{1}{\pi N}\E\mathrm{Im}\,\Tr\, G_{W,N}(E+\ii \eta).
    \end{equation}
    Therefore, the bound \eqref{eq_rho_W_N_Poisson_approximation} allows us to approximate the expected spectral density with the resolvent. Moreover, for every $x\in\ZN$, using \eqref{eq_xx_to_1_N_Tr}, \eqref{def_rho_W_N}, and the first bound in \eqref{eq_regularity_rho_W_N}, we obtain
    \begin{equation}\label{eq_E_Im_G_xx_bound}
        \begin{aligned}
            \E\mathrm{Im}\,G_{W,N}(E+\ii \eta)_{xx}=&\frac{1}{N}\E\mathrm{Im}\, \mathrm{Tr}\,G_{W,N}(E+\ii\eta)=\int_{\R}\frac{\eta}{(u-E)^2+\eta^2}\rho_{W,N}(u)\, \rd u\\
        \leq& C\int_{\R}\frac{\eta}{(u-E)^2+\eta^2}\, \rd u=C\pi.
        \end{aligned}
    \end{equation}

    We conclude this subsection by stating the fractional moment estimates and the local law.

    \begin{lemma}[Fractional moments]\label{lemma_fractional_moments}
        Suppose $H\equiv H_{W,N}$ is a random band matrix with bandwidth $W\geq 1$ and $N=nW$, defined in \Cref{def_considered_model}. For every fixed $s\in(0,1)$, there exists a constant $C_s\geq 1$ such that, for every deterministic real symmetric matrix $A\in\R^{\ZN\times\ZN}$,
        \begin{equation}\label{eq_fractional_flat_bound}
            \sup_{E\in\R}\E\absa{\q{(H+A-E)^{-1}}_{xy}}^s\leq C_{s}W^{s/2}.
        \end{equation}
        The bound \eqref{eq_fractional_flat_bound} remains valid with the $H+A$ on the left-hand side replaced by any principal minors of $H+A$ or $H^{\mathrm{o}}+A$.
        
        Fix a compact set $K\subseteq \R$. For the open-boundary form $H^{\mathrm{o}}$, there exist constants $c,q\in(0,1)$ and $C\geq 1$, depending only on $K$, such that, for all deterministic real symmetric matrices $U,V\in\R^{W\times W}$,
        \begin{equation}\label{eq_fractional_decay_bound_H_o}
            \sup_{E\in K}\E\norma{G^{U,V}(E)_{\qa{0}\qa{x}}}^q\leq CW^{C}\exp(-c|x|/W^2).
        \end{equation}
        Here, we denote $G^{U,V}(E):=(H^{U,V}-E)^{-1}$, and $H^{U,V}$ is defined by adding $U$ and $V$ to the $(\q{0},\q{0})$-block and $(\q{N-1},\q{N-1})$-block of $H^{\mathrm{o}}$, respectively. For the periodic form $H$, there exist constants $c,q\in(0,1)$ and $C\geq 1$, depending only on $K$, such that
        \begin{equation}\label{eq_fractional_decay_bound_H}
            \sup_{E\in K}\E\absa{G_{W,N}(E)_{xy}}^{q}\leq CW^{C}\exp\pa{-c|x-y|/W^2}.
        \end{equation}
    \end{lemma}

    \begin{proof}
        For the bound \eqref{eq_fractional_flat_bound}, given any deterministic real symmetric $A\in\R^{\ZN\times\ZN}$, if $x=y$, the Schur complement with respect to $H_{xx}$ yields
        \begin{equation}
            [(H+A-E)^{-1}]_{xx}=(H_{xx}-w_{x}\qa{(H_{ab}:(a,b)\neq (x,x))})^{-1},
        \end{equation}
        where $w_{x}\qa{\cdot}$ is a real function of variables $(H_{ab}:(a,b)\neq (x,x))$. Then we have
        \begin{equation*}
            \E\absa{[(H+A-E)^{-1}]_{xx}}^s\leq \E\E_{H_{xx}}\absa{H_{xx}-w_{x}\qa{(H_{ab}:(a,b)\neq (x,x))}}^{-s}\leq \sup_{w\in\R}\E|H_{xx}-w|^{-s}\leq C_sW^{s/2}.
        \end{equation*}
        The proof of the off-diagonal case is similar. For $x\neq y$, denote $u:= (H_{xx}+H_{yy})/2$ and $v:= (H_{xx}-H_{yy})/2$. Then, $u$ and $v$ are independent random variables. By the Schur complement, we have
        \begin{equation}
            \begin{aligned}
                [(H+A-E)^{-1}]_{\ha{x,y},\ha{x,y}}=&(H_{\ha{x,y},\ha{x,y}}-w_{x,y}\qa{(H_{ab}:(a,b)\notin \ha{x,y}\times\ha{x,y})})^{-1}\\
                =&\pa{u+\wh w[(H_{xy},v,(H_{ab}:(a,b)\notin \ha{x,y}\times\ha{x,y}))]}^{-1},
            \end{aligned}
        \end{equation}
        where $w_{x,y}$ and $\wh w$ are functions taking values in the space of real symmetric $2\times 2$ matrices. This gives
        \begin{equation}
            [(H+A-E)^{-1}]_{xy}=(u+\lambda_{x})^{-1}\bv_{x}(x)\ol{\bv}_{x}(y)+(u+\lambda_{y})^{-1}\bv_{y}(x)\ol{\bv}_{y}(y),\end{equation}
        where $\lambda_{x},\lambda_{y}$ are the eigenvalues of the matrix $\wh w[(H_{xy},v,(H_{ab}:(a,b)\notin \ha{x,y}\times\ha{x,y}))]$ and $\bv_{x},\bv_{y}$ are corresponding orthonormal eigenvectors. 
        Since $u$ is independent of $(H_{xy},v,(H_{ab}:(a,b)\notin \ha{x,y}\times\ha{x,y}))$, we have
        \begin{equation}
            \begin{aligned}
                &\E\absa{[(H+A-E)^{-1}]_{xy}}^s\leq \E\E_{u}|u+\lambda_{x}|^{-s}+\E\E_{u}|u+\lambda_{y}|^{-s} \leq 2 \sup_{\lambda\in\R}\E_{u}|u-\lambda|^{-s}\leq C_sW^{s/2}.
            \end{aligned}
        \end{equation}
        This completes the proof of \eqref{eq_fractional_flat_bound}.

The endpoint case of the bound \eqref{eq_fractional_decay_bound_H_o}, i.e., the case with $\qa{x}=\qa{N-1}$, follows from \cite[Theorem 2]{Localization1_2}. Specifically, we remove the endpoint deformations by the Schur complement. For any $0\leq i\leq j\leq n-1$, denote
\begin{equation}
	[i:j]\equiv[i:j]_{W}:=\bigcup_{k=i}^{j}\qa{kW}.
\end{equation}
Applying the Schur complement twice, we have
%\begin{equation}
\begin{align*}
	&G^{U,V}(E)_{\qa{0}\qa{N-1}}=-G^{U,V}(E)_{\qa{0}\qa{0}}H_{\qa{0}\qa{W}}\qb{\p{H^{U,V}_{[1:n-1]}-E}^{-1}}_{\qa{W}\qa{N-1}}\\
	=&G^{U,V}(E)_{\qa{0}\qa{0}}H_{\qa{0}\qa{W}}\qb{\p{H^{\mathrm{o}}_{[2:n-2]}-E}^{-1}}_{\qa{W}\qa{N-W-1}}H_{\qa{N-W-1}\qa{N-1}}\qb{\p{H^{U,V}_{[1:n-1]}-E}^{-1}}_{\qa{N-1}\qa{N-1}},
\end{align*}
%\end{equation}
which, together with a direct application of H{\" o}lder's inequality, immediately yields for any $q\in(0,1/25]$,
%\begin{equation}
\begin{align}
	&\E\norm{G^{U,V}(E)_{\qa{0}\qa{N-1}}}^q\leq \pb{\E\normb{G^{U,V}(E)_{\qa{0}\qa{0}}}^{5q}}^{1/5}\pa{\E\norm{H_{\qa{0}\qa{W}}}^{5q}}^{1/5}\pb{\E\normb{\qb{\p{H^{\mathrm{o}}_{[2:n-2]}-E}^{-1}}_{\qa{W}\qa{N-W-1}}}^{5q}}^{1/5}\nonumber\\
	&\qquad\qquad\qquad\qquad\quad\,\times\pa{\E\norm{H_{\qa{N-W-1}\qa{N-1}}}^{5q}}^{1/5}\pb{\E\normb{\qb{\p{H^{U,V}_{[1:n-1]}-E}^{-1}}_{\qa{N-1}\qa{N-1}}}^{5q}}^{1/5}.\label{eq_G_U_V_0_N_1_decomposition_bound}
\end{align}
%\end{equation}
Applying \eqref{eq_fractional_flat_bound} and using the fact that the operator norm is bounded by the Hilbert-Schimidt norm, we bound the first and the last factors on the right-hand side of \eqref{eq_G_U_V_0_N_1_decomposition_bound} as
\begin{equation}
	\E\normb{G^{U,V}(E)_{\qa{0}\qa{0}}}^{5q}+\E\normb{\qb{\p{H^{U,V}_{[1:n-1]}-E}^{-1}}_{\qa{N-1}\qa{N-1}}}^{5q}\leq CW^{C}.
\end{equation}
By standard operator norm bound of random matrices, the second and the fourth factors on the right-hand side of \eqref{eq_G_U_V_0_N_1_decomposition_bound} are bounded as
\begin{equation}
	\E\norm{H_{\qa{0}\qa{W}}}^{5q}\leq \pa{\E\norm{H_{\qa{0}\qa{W}}}}^{5q}\leq C_{q},\qquad \E\norm{H_{\qa{N-W-1}\qa{N-1}}}^{5q}\leq \pa{\E\norm{H_{\qa{N-W-1}\qa{N-1}}}}^{5q}\leq C_q.
\end{equation}
Finally, for the third factor on the right-hand side of \eqref{eq_G_U_V_0_N_1_decomposition_bound}, applying \cite[Theorem 2]{Localization1_2} and bounding the operator norm by the Hilbert-Schimidt norm, we obtain
\begin{equation}
	\E\normb{\qb{\p{H^{\mathrm{o}}_{[2:n-1]}-E}^{-1}}_{\qa{W}\qa{N-W-1}}}^{5q}\leq CW^{C}\exp[-c(N-2W-1)/W^2]\leq CW^{C}\exp[-c'N/W^2].
\end{equation}
Here, in the above three estimates, $C,C_q,c,c'$ are all constants. Therefore, plugging these estimates back into \eqref{eq_G_U_V_0_N_1_decomposition_bound} yields the desired endpoint case of \eqref{eq_fractional_decay_bound_H_o}.

        To extend the estimate \eqref{eq_fractional_decay_bound_H_o} to general $x\in\ZN$, it suffices to consider the case $\qa{x}=\qa{\ell W}$ for $\ell\in\ha{0,1,\ldots,n-1}$. Moreover, the case $\ell=n-1$ follows from the endpoint case $\qa{x}=\qa{N-1}$. Fix an $\ell\in \ha{0,1,\ldots,n-2}$, and denote
        \begin{equation}
            I:=\bigcup_{k=0}^{\ell}\qa{k W}.
        \end{equation}
        Then, the Schur complement provides
        \begin{equation}
            \begin{aligned}
                G^{U,V}(E)_{\qa{0}\qa{x}}=\qa{G^{U,V}(E)_{II}}_{\qa{0}\qa{x}}=\qbb{\pa{H_{II}^{U,V}-E-H^{\mathrm{o}}_{II^{\txt{c}}}\pa{H^{U,V}_{I^{\txt{c}}I^{\txt{c}}}-E}^{-1}H_{I^{\txt{c}}I}^{\mathrm{o}}}^{-1}}_{\qa{0}\qa{x}}.
            \end{aligned}
        \end{equation}
        By the definition of $H$, the matrix
        \begin{equation}
            H^{\mathrm{o}}_{II^{\txt{c}}}\pa{H^{U,V}_{I^{\txt{c}}I^{\txt{c}}}-E}^{-1}H_{I^{\txt{c}}I}^{\mathrm{o}}
        \end{equation}
         is supported on the $\qa{\ell W}\qa{\ell W}$ block and is independent of $H_{II}^{U,V}$. Therefore, the endpoint estimate of the matrix $H_{II}^{U,V}$ on $I$ yields the bound \eqref{eq_fractional_decay_bound_H_o}.\footnote{Strictly speaking, \eqref{eq_fractional_decay_bound_H_o} does not apply when $\ell<99$ since we assume $n\geq 100$ in \Cref{def_considered_model}. But this subtlety can handled by applying the uniform rough bound \eqref{eq_fractional_flat_bound} when $\ell<99$.}

        Finally, for the bound \eqref{eq_fractional_decay_bound_H}, when $\qa{x}=\qa{y}$, it readily follows from \eqref{eq_fractional_flat_bound}. For the case $\qa{x}\neq \qa{y}$, we apply the resolvent identity
        \begin{equation}
            G(E)_{\qa{y}^{\txt{c}}\qa{y}}=-(H^{\mathrm{o}}_{\qa{y}^{\txt{c}}\qa{y}^{\txt{c}}}-E)^{-1}H_{\qa{y}^{\txt{c}}\qa{y}}G(E)_{\qa{y}\qa{y}},
        \end{equation}
        which yields
        \begin{equation}
            G(E)_{xy}=-\sum_{a\in\qa{y}^{\txt{c}}}\sum_{b\in\qa{y}}\qb{(H^{\mathrm{o}}_{\qa{y}^{\txt{c}}\qa{y}^{\txt{c}}}-E)^{-1}}_{xa}H_{ab}G(E)_{by}.
        \end{equation}
        Here, to highlight that this matrix is of open-boundary form, we refer to the restriction of $H$ to $\qa{y}^{\txt{c}}$ as $H^{\mathrm{o}}_{\qa{y}^{\txt{c}}\qa{y}^{\txt{c}}}$. 
        Taking $s=q/3\in(0,1/3)$, where $q$ is the exponent in \eqref{eq_fractional_decay_bound_H_o}, and applying H{\" o}lder's inequality, we obtain
\begin{align*}
                &\E\absa{G(E)_{xy}}^s\leq \sum_{a\in\qa{y}^{\txt{c}}}\sum_{b\in\qa{y}}\pB{\E\absa{\qb{(H^{\mathrm{o}}_{\qa{y}^{\txt{c}}\qa{y}^{\txt{c}}}-E)^{-1}}_{xa}}^{3s}}^{1/3}\pa{\E|H_{ab}|^{3s}}^{1/3}\pa{\E|G(E)_{by}|^{3s}}^{1/3}\\
                &\leq CW^{C}\sum_{a\in\qa{y}^{\txt{c}}:|\qa{a}-\qa{y}|\leq 1}\sum_{b\in\qa{y}}\pB{\E\absa{\qb{(H^{\mathrm{o}}_{\qa{y}^{\txt{c}}\qa{y}^{\txt{c}}}-E)^{-1}}_{xa}}^{3s}}^{1/3}\leq C'W^{C'}\exp(-c|x-y|/W^2).
            \end{align*}
Here, $C,C',c$ are constants, and we also used the bounds $\E|H_{ab}|^{3s}\leq 100W^{-3s/2}\mathbf{1}_{|\qa{a}-\qa{b}|\leq 1}$ and \eqref{eq_fractional_flat_bound} in the second step, together with the bound \eqref{eq_fractional_decay_bound_H_o} in the third step. This yields the bound \eqref{eq_fractional_decay_bound_H} and hence completes the proof of \Cref{lemma_fractional_moments}.
\end{proof}

    By the maximum principle, the fractional moment estimate \eqref{eq_fractional_decay_bound_H} can be immediately extended to the resolvents with complex spectral parameters.
    \begin{lemma}\label{lemma_fractional_decay_bound_H_E_i_eta}
        Under the assumptions of \Cref{lemma_fractional_moments}, for every fixed compact set $K\subseteq \R$, there exist constants $c,q\in(0,1)$ and $C\geq 1$, depending only on $K$, such that
        \begin{equation}\label{eq_fractional_decay_bound_H_E_i_eta}
            \sup_{E\in K}\sup_{\eta\in[0,1]}\E\absa{G_{W,N}(E+\ii\eta)_{xy}}^{q}\leq CW^{C}\exp\pa{-c|x-y|/W^2}.
        \end{equation}
    \end{lemma}

    \begin{proof}
        Suppose that $K+[-1,1]\subseteq J$, where $J\subseteq \R$ is a finite closed interval, and let $q\in(0,1)$ be the constant in \eqref{eq_fractional_decay_bound_H} with $K$ replaced by $J$. We define
        \begin{equation}
            u(z):=\log\E\absa{G_{W,N}(z)_{xy}}^{q},\qquad \forall z\in\C_+.
        \end{equation}
        Clearly, $u(z)$ is subharmonic on the upper half-plane. Moreover, taking $s\in(q,1)$ in \eqref{eq_fractional_flat_bound} and using the fact that $z\mapsto \E\absa{G_{W,N}(E+\ii \eta)_{xy}}^{s}$ is subharmonic on $\C_+$, we have
        \begin{equation}
            \E\absa{G_{W,N}(E+\ii \eta)_{xy}}^{s}\leq \int_{\R}P_{\eta}(E-\lambda)\E\absa{G_{W,N}(\lambda)_{xy}}^{s}\, \rd \lambda\leq C_sW^{s/2},\qquad \forall E+\ii \eta\in\C_+.
        \end{equation}
        This higher moment bound, together with the bound $\E\absa{G_{W,N}(E)_{xy}}^{s}\leq C_sW^{s/2}$ from \eqref{eq_fractional_flat_bound}, ensures that $u(z)$ extends continuously to $\R$ as $u(E):=\log\E\absa{G_{W,N}(E)_{xy}}^{q}$. By \eqref{eq_fractional_flat_bound} and \eqref{eq_fractional_decay_bound_H}, on the boundary $\partial\C_+=\R$, for constants $C\geq 1$ and $c\in(0,1)$, the function $u$ is bounded as
        \begin{equation}
            u(E)\leq
            \begin{cases}
                C(1+\log W), & E\notin J,\\
                C(1+\log W)-c|x-y|/W^2, & E\in J.
            \end{cases}
        \end{equation}
        Then, for every $E\in K$ and $\eta\in(0,1]$ the two-constant theorem in \cite[Theorem 4.3.7]{Ransford_1995} yields
        \begin{equation}
            \begin{aligned}
                u(z)\leq& C(1+\log W)-c\frac{|x-y|}{\pi W^2}\int_{J}\frac{\eta}{(E-\lambda)^2+\eta^2}\, \rd \lambda\\
                \leq& C(1+\log W)-c\frac{|x-y|}{\pi W^2}\int_{|\lambda|\leq \eta}\frac{\eta}{\lambda^2+\eta^2}\, \rd \lambda= C(1+\log W)-c\frac{|x-y|}{2 W^2},
            \end{aligned}
        \end{equation}
        where we also used the fact that $[E-1,E+1]\subseteq K+[-1,1]\subseteq J$ in the second step. Together with \eqref{eq_fractional_decay_bound_H}, this provides, uniformly in $E\in K$ and $\eta\in[0,1]$,
        \begin{equation}
            \E\absa{G_{W,N}(E+\ii\eta)_{xy}}^{q}\leq CW^{C}\exp\pa{-c|x-y|/W^2},
        \end{equation}
        where $C\geq 1$ is a large constant and $c\in (0,1)$ is a small constant. Taking the supremum over $E\in K$ and $\eta\in[0,1]$ yields the fractional moment bound \eqref{eq_fractional_decay_bound_H_E_i_eta} and completes the proof of \Cref{lemma_fractional_decay_bound_H_E_i_eta}.
\end{proof}

    The following local law from \cite{erdos2025zigzagstrategyrandomband} will serve as the key input for controlling the density of states.

    \begin{lemma}[Local law in the delocalized phase]\label{lemma_local_law}
        Consider the random band matrix $H\equiv H_{W,N}$ defined in \Cref{def_considered_model}, fix any small constants $\kappa,c>0$, and suppose that $W^2\geq N^{1+c}$. Then, for any fixed small constants $\tau,\varepsilon>0$ and large constants $D,D'\geq 1$, there exists a large integer $W_0\equiv W_0\pa{\kappa,c,\tau,\varepsilon,D,D'}\geq 1$, such that, for every $W\geq W_0$,
        \begin{equation}\label{eq_local_law}
            \P\pa{\bigcap_{|E|\leq 2-\kappa}\bigcap_{N^{-1+\tau}\leq \eta\leq 1}\bigcap_{x,y\in\ZN}\ha{|G(E+\ii \eta)_{xy}-m_{\txt{sc}}\delta_{xy}|^2\leq \frac{W^{\varepsilon}}{\ell_{\eta}\eta}\pa{\frac{|x-y|}{\ell_{\eta}}+1}^{-D}}}\geq 1-W^{-D'}.
        \end{equation}
        Here, the characteristic length is defined by
        \begin{equation}
            \ell_\eta:=\min\pa{W\eta^{-1/2},N}.
        \end{equation}
    \end{lemma}

    \begin{proof}
        This follows directly from (3.17) in Theorem 3.3 and the case (ii) of Proposition 2.8 in \cite{erdos2025zigzagstrategyrandomband}.
    \end{proof}

    \subsection{Density of states}\label{sec_density_of_states}
    In this subsection, we establish some properties of the expected spectral density. These properties complement the estimates in \Cref{sec_a_prioro_estimates}, and finally lead to the proof of asymptotic Poisson statistics stated in \Cref{theorem_Poisson_process}. Specifically, this subsection is devoted to proving the following lemma.

    \begin{lemma}\label{lemma_rho_W_N}
        Suppose $H\equiv H_{W,N}$ is a random band matrix with bandwidth $W\geq 1$ and $N=nW$, defined in \Cref{def_considered_model}. Let $\rho_{W,N}$ be the unique smooth density defined by \eqref{def_rho_W_N} (recall \Cref{lemma_rho_W_N_regularity}). For each fixed $W\geq 1$, there exists a continuous strictly positive probability density $\rho_{W}$ on $\R$, such that $\rho_{W,N}$ converges to $\rho_{W}$ locally uniformly as $N\to \infty$. More specifically, for every compact $K\subseteq \R$, there exist a $K$-dependent small constant $c>0$ and $K$-dependent large constants $C,D\geq 1$, such that
        \begin{equation}\label{eq_rho_W_N_rho_W}
            \sup_{E\in K}\absa{\rho_{W,N}(E)-\rho_{W}(E)}\leq C(NW)^{C}\exp\pa{-cN/W^2},\qquad \forall N\geq DW^2.
        \end{equation}
        
        In the bulk regime, for every compact $K\subseteq (-2,2)$, there exist a $K$-dependent small constant $\delta>0$, and $K$-dependent large constants $W_0, C\geq 1$, such that
        \begin{equation}\label{eq_rho_W_to_rho_sc}
            \sup_{E\in K}\absa{\rho_W(E)-\rho_{\txt{sc}}(E)}\leq CW^{-\delta},\qquad \forall W\geq W_0.
        \end{equation}
        Furthermore, the family $\ha{\rho_W}_{W=1}^{\infty}$ is equicontinuous on $K$, and there exists a large constant $C\equiv C_{K}\geq 1$, such that
        \begin{equation}\label{eq_rho_W_uniform_bounds}
            C^{-1}\leq \inf_{W\geq 1} \inf_{E\in K}\rho_{W}(E)\leq \sup_{W\geq 1} \sup_{E\in K}\rho_{W}(E)\leq C.
        \end{equation}
\end{lemma}

    \begin{proof}
        First, for the estimate \eqref{eq_rho_W_N_rho_W}, we claim the following comparison estimate, whose proof relies on a coupling argument and is postponed to \Cref{sec_proof_of_lemma_comparison}.
        \begin{lemma}\label{lemma_N_N_comparison}
            Under the assumptions of \Cref{lemma_rho_W_N}, fix any compact $K\subseteq \R$. Then, there exist a small constant $c\equiv c_{K}>0$ and a large constant $C\equiv C_{K}\geq 1$, such that, for any $N\leq N'\leq 2N$, the following estimate holds uniformly in $E\in K$ and $\eta\in(0,1]$:
            \begin{equation}\label{eq_Tr_G_comparison}
                \absa{\frac{1}{N}\E\Tr\, G_{W,N}(E+\ii \eta)-\frac{1}{N'}\E\Tr\, G_{W,N'}(E+\ii \eta)}\leq CW^{C}\eta^{-2}\exp(-cN/W^2).
            \end{equation}
            Moreover, for the expected spectral density, there exists a large constant $D\equiv D_{K}\geq 1$, such that, for any $N\leq N'\leq 2N$,
            \begin{equation}\label{eq_rho_W_N_rho_W_N_comparison}
                \sup_{E\in K}\absa{\rho_{W,N}(E)-\rho_{W,N'}(E)}\leq C(NW)^C\exp(-cN/W^2),\qquad \forall N\geq DW^2.
            \end{equation}
            Here, $c\equiv c_{K}\in(0,1)$ and $C\equiv C_{K}\geq 1$ are again $K$-dependent constants, as in \eqref{eq_Tr_G_comparison}.
        \end{lemma}
        With \Cref{lemma_N_N_comparison}, we can immediately obtain the convergence of $\rho_{W,N}$, i.e., the bound \eqref{eq_rho_W_N_rho_W}. Indeed, for any $N'\geq N\geq DW^2$, there exists some $k\in\Z_+$, such that
        \begin{equation}
            2^{k-1}N\leq N'\leq 2^k N.
        \end{equation}
        We denote $\norm{f}_{K}=\sup_{E\in K}|f(E)|$ and write
        \begin{equation*}
            \sup_{E\in K}\absa{\rho_{W,N}(E)-\rho_{W,N'}(E)}=\norma{(\rho_{W,N}-\rho_{W,2N})+\cdots+(\rho_{W,2^{k-2}N}-\rho_{W,2^{k-1}N})+(\rho_{W,2^{k-1}N}-\rho_{W,N'})}_{K}.
        \end{equation*}
        Then, applying the triangle inequality and the bound \eqref{eq_rho_W_N_rho_W_N_comparison}, we obtain
        \begin{equation}
            \begin{aligned}
                \sup_{E\in K}\absa{\rho_{W,N}(E)-\rho_{W,N'}(E)}\leq& C\sum_{k=1}^{\infty}(2^{k-1}NW)^{C}\exp(-c\cdot 2^{k-1}N/W^2)\\
            \leq& C'(NW)^{C'}\exp(-c'N/W^2),
            \end{aligned}
        \end{equation}
        which shows that, for any $W\geq 1$, the sequence $\h{\rho_{W,N}}_{N/W=1}^{\infty}$ is Cauchy, with respect to the norm $\|\cdot\|_{K}$. Therefore, for every $W\geq 1$, there exists a non-negative continuous function $\rho_W$ on $\R$, such that the sequence $\h{\rho_{W,N}}_{N/W=1}^{\infty}$ converges locally uniformly to $\rho_W$, and the bound \eqref{eq_rho_W_N_rho_W} holds. Moreover, the fact that, for every $W$,
        \begin{equation}
            \int_{\R}E^2\rho_{W,N}(E)\, \rd E=\frac{1}{N}\E\Tr\, H_{W,N}^2=1,\qquad \forall N,W,
        \end{equation}
        shows that the family $\h{\rho_{W,N}}_{N,W}$ is uniformly tight. Therefore, together with the locally uniform convergence, this implies that $\rho_{W}$ is a probability density. Next, we prove that $\rho_{W}$ is strictly positive on $\R$. To this end, given any $x\in\ZN$, applying the Schur complement at $H_{xx}$, we can obtain the identity
        \begin{equation}\label{eq_rho_W_N_E_p_W}
            \rho_{W,N}(E)=\E\, p_W\qbb{E+\sum_{i,j\in\ZN\setminus\ha{x}}H_{xi}\q{(H^{(x)}-E)^{-1}}_{ij}H_{jx}},\qquad \forall E\in\R.
        \end{equation}
        Here, $p_W(u):=\sqrt{3W/2\pi}\exp(-3Wu^2/2)$ is the probability density of $H_{xx}$, and $H^{(x)}$ is the minor obtained by deleting the $x$-th row and column. To see \eqref{eq_rho_W_N_E_p_W}, for any $\eta>0$, denoting $z=E+\ii \eta$, by the Schur complement, we have
        \begin{equation}
            \mathrm{Im}\, G_{W,N}(E+\ii \eta)_{xx}=\frac{\eta+\beta_{\eta}}{\p{H_{xx}-E-\alpha_{\eta}}^2+\p{\eta+\beta_{\eta}}^2},
        \end{equation}
        where the real random variables $\alpha_{\eta}\equiv \alpha_{\eta}(x,H)$ and $\beta_{\eta}\equiv \beta_{\eta}(x,H)$ are defined by
        \begin{equation}
            \alpha_{\eta}+\ii \beta_{\eta}:=\sum_{i,j\in\ZN\setminus\ha{x}}H_{xi} \q{(H^{(x)}-E-\ii \eta)^{-1}}_{ij}H_{jx}.
        \end{equation}
        Taking expectation with respect to $H_{xx}$, we obtain
        \begin{equation}
            \begin{aligned}
                \E_{H_{xx}}\qa{\mathrm{Im}\, G_{W,N}(z)}_{xx}=&\int_{\R}\frac{\eta+\beta_{\eta}}{\p{u-E-\alpha_{\eta}}^2+\p{\eta+\beta_{\eta}}^2}p_W(u)\, \rd u\\
                =&\int_{\R}\frac{1}{1+u^2}p_W\qa{E+\alpha_{\eta}+(\eta+\beta_{\eta})u}\, \rd u,
            \end{aligned}
        \end{equation}
        which immediately yields
        \begin{equation}
            \begin{aligned}
                \pi P_{\eta}*\rho_{W,N}(E)=&\frac{1}{N}\E\Tr\,\mathrm{Im}\, G_{W,N}(E+\ii \eta)=\E\mathrm{Im}\, G_{W,N}(z)_{xx}=\E\, \E_{H_{xx}}\qa{\mathrm{Im}\, G_{W,N}(z)}_{xx}\\
                =&\E\int_{\R}\frac{1}{1+u^2}p_W\qa{E+\alpha_{\eta}+(\eta+\beta_{\eta})u}\, \rd u,
            \end{aligned}
        \end{equation}
        where we also used \eqref{eq_xx_to_1_N_Tr} in the second step. Letting $\eta\to 0$ on both sides, by Lebesgue's dominated convergence theorem, we obtain
        \begin{equation}
            \begin{aligned}
                \pi \rho_{W,N}(E)=&\E\, p_W\qbb{E+\sum_{i,j\in\ZN\setminus\ha{x}}H_{xi}\qb{(H^{(x)}-E)^{-1}}_{ij}H_{jx}}\int_{\R}\frac{1}{1+u^2}\, \rd u\\
                =&\pi\E\, p_W\qbb{E+\sum_{i,j\in\ZN\setminus\ha{x}}H_{xi}\qb{(H^{(x)}-E)^{-1}}_{ij}H_{jx}}.
            \end{aligned}
        \end{equation}
        Moreover, the fractional moment bound \eqref{eq_fractional_flat_bound} shows, for any fixed $s\in(0,1)$,
        \begin{equation}
            \begin{aligned}
                \E\absBB{\sum_{i,j\in\ZN\setminus\ha{x}}H_{xi}\qb{(H^{(x)}-E)^{-1}}_{ij}H_{jx}}^s&\leq \sum_{i,j\in\ZN\setminus\ha{x}}\E |H_{xi}|^s|H_{jx}|^s\E\absa{\qb{(H^{(x)}-E)^{-1}}_{ij}}^s\\
                &\leq C_sW^{s/2}\sum_{i,j\in\ZN\setminus\ha{x}}\E|H_{xi}|^s|H_{jx}|^s\leq C_s'W^{2-s/2}.
            \end{aligned}
        \end{equation}
        Here, we also used the independence of entries of $H$. 
        Therefore, by Markov's inequality, for fixed $W$, there exists a large constant $C\geq 1$, such that the random variable $\sum_{i,j\in\ZN\setminus\ha{x}}H_{xi}\qb{(H^{(x)}-E)^{-1}}_{ij}H_{jx}$ is bounded by $CW^{C}$ with probability at least $1/2$. For fixed $W\geq 1$, combining this fact with \eqref{eq_rho_W_N_E_p_W} yields
        \begin{equation}
            \rho_{W,N}(E)\geq \frac{1}{2}\inf_{|u-E|\leq CW^{C}}p_W(u),
        \end{equation}
        which further yields the strict positivity of $\rho_{W}(E)$ by taking the limit $N\to \infty$ on both sides.

        Next, we prove the remainder of \Cref{lemma_rho_W_N}. In fact, it suffices to prove the estimate \eqref{eq_rho_W_to_rho_sc}. Once we have \eqref{eq_rho_W_to_rho_sc}, the Arzel\`a-Ascoli theorem will immediately imply the equicontinuity, and the lower and upper bounds of $\rho_{\txt{sc}}$ on $K$ will yield
        \begin{equation}\label{eq_rho_W_uniform_bounds_W_0}
            C^{-1}\leq \inf_{W\geq W_0} \inf_{E\in K}\rho_{W}(E)\leq \sup_{W\geq W_0} \sup_{E\in K}\rho_{W}(E)\leq C
        \end{equation}
        for some $K$-dependent large constants $W_0,C\geq 1$. Since $\rho_W$ is strictly positive and continuous, the bound \eqref{eq_rho_W_uniform_bounds_W_0} immediately gives the bound \eqref{eq_rho_W_uniform_bounds}. Therefore, for the rest of this proof, we fix a compact $K\subseteq (-2,2)$ and consider the difference $\rho_{W}(E)-\rho_{\txt{sc}}(E)$. To this end, for any $\eta\in(0,1/2)$ and any positive multiple $M$ of $W$, we write
        \begin{equation}
            \begin{aligned}
                &\absa{\rho_W(E)-\rho_{\txt{sc}}(E)}\leq \absa{\rho_W(E)-\rho_{W,M}(E)}+\absa{\rho_{W,M}(E)-P_{\eta}*\rho_{W,M}(E)}\\
                &+\absa{P_{\eta}*\rho_{W,M}(E)-\pi^{-1}\mathrm{Im}\,m_{\txt{sc}}(E+\ii \eta)}+\absa{\pi^{-1}\mathrm{Im}\,m_{\txt{sc}}(E+\ii \eta)-\rho_{\txt{sc}}(E)}=:\sum_{k=1}^{4}\absa{I_{W,M}^{(k)}(\eta,E)}.
            \end{aligned}
        \end{equation}
        To bound these four terms, we take the parameters
        \begin{equation}
            \eta=W^{-7/4},\qquad M=W\ceil{W^{5/4}}.
        \end{equation}
        Then, for sufficiently large $W$, by \eqref{eq_rho_W_N_rho_W}, the first term $I_{W,M}^{(1)}(\eta,E)$ is bounded as
        \begin{equation}\label{eq_bound_I_W_M_1}
            \sup_{E\in K}\absa{I_{W,M}^{(1)}(\eta,E)}=\sup_{E\in K}\absa{\rho_W(E)-\rho_{W,M}(E)}\leq C(MW)^{C}\exp\pa{-cM/W^2}\leq C_{K}W^{-2}.
        \end{equation}
        For the second term $I_{W,M}^{(2)}(\eta,E)$, applying \eqref{eq_rho_W_N_Poisson_approximation}, we obtain, for sufficiently large $W$,
        \begin{equation}\label{eq_bound_I_W_M_2}
            \sup_{E\in K}\absa{I_{W,M}^{(2)}(\eta,E)}\leq C(MW)^{1/2}\eta\log\eta^{-1}+C\eta\leq C_{K}W^{-1/9}.
        \end{equation}
        For the fourth term $I_{W,M}^{(4)}(\eta,E)$, since $\rho_{\txt{sc}}(E)=\pi^{-1}\mathrm{Im}\, m_{\txt{sc}}(E+\ii 0)$, the Newton-Leibniz formula yields
        \begin{equation}\label{eq_I_4_Newton_Leibniz}
            I_{W,M}^{(4)}(\eta,E)=\frac{1}{\pi}\mathrm{Im}\,\int_{0}^{\eta}\frac{\rd}{\rd u}m_{\txt{sc}}(E+\ii u)\, \rd u=\frac{1}{\pi}\mathrm{Re}\,\int_{0}^{\eta} \frac{m_{\txt{sc}}(E+\ii u)^2}{1-m_{\txt{sc}}(E+\ii u)^2}\, \rd u.
        \end{equation}
        By standard calculus, there exists a small constant $c_{K}$, such that
        \begin{equation}
            \inf_{E\in K}\inf_{\eta\in(0,1)}\absa{1-m_{\txt{sc}}(E+\ii\eta)^2}\geq c_{K}.
        \end{equation}
        Substituting this back into \eqref{eq_I_4_Newton_Leibniz}, we then obtain, for every sufficiently large $W$ and some large constant $C_{K}\geq 1$,
        \begin{equation}\label{eq_bound_I_W_M_4}
            \sup_{E\in K}\absa{I_{W,M}^{(4)}(\eta,E)}\leq C_{K}\eta= C_{K}W^{-7/4}.
        \end{equation}
        It remains to bound the third term $I_{W,M}^{(3)}(\eta,E)$ by the local law in \Cref{lemma_local_law}. However, since $M\asymp W^{9/4}\geq W^2$ for large $W$, \Cref{lemma_local_law} does not directly provide the required control of $I_{W,M}^{(3)}(\eta,E)$ at this scale. Therefore, we introduce the intermediate scale $M_0=W\ceil{W^{15/16}}$ and further write
        \begin{equation}
            I_{W,M}^{(3)}(\eta,E)=I_{W,M}^{(3,1)}(\eta,E)+I_{W,M}^{(3,2)}(\eta,E),
        \end{equation}
        where we denote
        \begin{equation}
            I_{W,M}^{(3,1)}(\eta,E):=P_{\eta}*\rho_{W,M}(E)-P_{\eta}*\rho_{W,M_0}(E),\quad I_{W,M}^{(3,2)}(\eta,E):=P_{\eta}*\rho_{W,M_0}(E)-\pi^{-1}\mathrm{Im}\,m_{\txt{sc}}(E+\ii \eta).
        \end{equation}
        For sufficiently large $W$, the term $I_{W,M}^{(3,2)}(\eta,E)$ can be directly bounded by the local law \eqref{eq_local_law} as
        \begin{equation}\label{eq_bound_I_W_M_3_2}
            \pi \absa{I_{W,M}^{(3,2)}(\eta,E)}=\frac{1}{M_0}\absa{\E\mathrm{Im}\,\Tr\, \qa{G_{W,M_0}(E+\ii \eta)-m_{\txt{sc}}(E+\ii \eta)}}\leq \frac{CW^{1/64}}{\sqrt{\ell_{\eta}\eta}}+\frac{C}{\eta}W^{-100}\leq CW^{-1/32}.
        \end{equation}
        Here, we take $D'=100$ in \eqref{eq_local_law}, and estimate $G_{W,M_0}(E+\ii \eta)$ and $m_{\txt{sc}}(E+\ii \eta)$ by the trivial bound $\eta^{-1}$ on the bad event. Finally, for the term $I_{W,M}^{(3,1)}(\eta,E)$, in fact, we have the following stronger comparison estimate.
        \begin{lemma}\label{lemma_rho_M_1_M_0_comparison}
            Under the assumptions of \Cref{lemma_rho_W_N}, adopt the notations in this proof. Then, for any fixed large constant $D\geq 1$, there exists a large constant $C\equiv C_{K,D}\geq 1$, such that
            \begin{equation}\label{eq_rho_M_1_rho_M_0_comparison}
                \sup_{M_1\in W\Z_+: M_1\geq 2M_0}\sup_{E\in K}\absa{P_\eta * \rho_{W,M_1}(E)-P_\eta * \rho_{W,M_0}(E)}\leq CW^{-D}.
            \end{equation}
        \end{lemma}
        The proof of this lemma relies on a coupling argument similar to that for \Cref{lemma_N_N_comparison}, and is therefore also postponed to \Cref{sec_proof_of_lemma_comparison}. 
        With \Cref{lemma_rho_M_1_M_0_comparison}, taking $M_1=M$ in \eqref{eq_rho_M_1_rho_M_0_comparison} and combining the estimates \eqref{eq_bound_I_W_M_1}, \eqref{eq_bound_I_W_M_2}, \eqref{eq_bound_I_W_M_4}, and \eqref{eq_bound_I_W_M_3_2}, we obtain the desired estimate \eqref{eq_rho_W_to_rho_sc}. This completes the proof of \Cref{lemma_rho_W_N}.
    \end{proof}

    \subsection{Proof of \texorpdfstring{Lemmas \ref{lemma_N_N_comparison} and \ref{lemma_rho_M_1_M_0_comparison}}{comparison estimates}}\label{sec_proof_of_lemma_comparison}
    We begin with the proof of \Cref{lemma_N_N_comparison}.
    \begin{proof}[Proof of \Cref{lemma_N_N_comparison}]
        We first show how the bound \eqref{eq_rho_W_N_rho_W_N_comparison} can be derived from the bound \eqref{eq_Tr_G_comparison}. Given any fixed sufficiently small constant $c_0\in(0,1)$, taking $\eta=\exp(-c_0N/W^2)$ in \eqref{eq_Tr_G_comparison}, and applying the approximation \eqref{eq_rho_W_N_Poisson_approximation}, we obtain uniformly in $E\in K$
        \begin{equation}
            \begin{aligned}
                &\absa{\rho_{W,N}(E)-\rho_{W,N'}(E)}\\
                \leq& \absa{(\pi N)^{-1}\E\Tr\, \mathrm{Im}\, G_{W,N}(E+\ii \eta)-\rho_{W,N}(E)}+\absa{(\pi N')^{-1}\E\Tr\, \mathrm{Im}\, G_{W,N'}(E+\ii \eta)-\rho_{W,N'}(E)}\\
                &+\absa{(\pi N)^{-1}\E\Tr\, \mathrm{Im}\, G_{W,N}(E+\ii \eta)-(\pi N')^{-1}\E\Tr\, \mathrm{Im}\, G_{W,N'}(E+\ii \eta)}\\
                \leq& C(NW)^{1/2}\eta\log\eta^{-1}+C\eta+C(NW)^C\exp(-cN/W^2)\leq C'(NW)^{C'}\exp(-c'N/W^2).
            \end{aligned}
        \end{equation}
        Here, we take the constant $D$ in \eqref{eq_rho_W_N_rho_W_N_comparison} to be large enough to ensure $\eta\leq \exp(-c_0 D)< 1/2$, so that \eqref{eq_rho_W_N_Poisson_approximation} can be applied properly. 
        It remains to prove \eqref{eq_Tr_G_comparison}. In the remainder of this proof, we fix $E\in K$, $\eta\in(0,1]$, and denote $R=2\ceil{N/4W}W$. Then, clearly, it suffices to prove
        \begin{equation}\label{eq_Tr_G_N_M_comparison}
            \begin{aligned}
                \absa{\frac{1}{N}\E\Tr\, G_{W,N}(E+\ii \eta)-\frac{1}{R}\E\Tr\, G_{W,R}(E+\ii \eta)}\leq& CW^{C}\eta^{-2}\exp(-cN/W^2),\\
                \absa{\frac{1}{N'}\E\Tr\, G_{W,N'}(E+\ii \eta)-\frac{1}{R}\E\Tr\, G_{W,R}(E+\ii \eta)}\leq& CW^{C}\eta^{-2}\exp(-cN/W^2),
            \end{aligned}
        \end{equation}
We only prove the first bound here, since the proof of the second one is almost the same, with only slight notational modifications. Let $I\equiv I_{W,N}\subseteq \ZN$ be any set of indices consisting of $2\ceil{N/4W}$ successive blocks, and $\sx$ be the center of $I$, i.e., the $\ceil{N/4W}W$-th element in $I$. Denote the boundary blocks of $I$ and its complement by $\qa{x_1},\qa{x_2}$ and $\qa{y_1},\qa{y_2}$, respectively, where the blocks $\qa{x_i}$ and $\qa{y_i}$ are adjacent for $i=1,2$. Then, we have
        \begin{equation}\label{eq_dist_sx_boundary}
            \mathrm{dist}\pa{\sx,\qa{x_1}\cup\qa{x_2}\cup\qa{y_1}\cup\qa{y_2}}\geq N/100.
        \end{equation}
        We also denote the restrictions of $H$ to $I$ and its complement as $H_{I}^{\mathrm{o}}$ and $H_{I^{\mathrm{c}}}^{\mathrm{o}}$, respectively, where the superscript $\mathrm{o}$ means both these two matrices are of open-boundary form. Define
        \begin{equation}
            \wh H_{W,N}:=H_{I}\oplus H_{I^{\mathrm{c}}},
        \end{equation}
        where $H_{I}$ is the periodic form of $H_{I}^{\mathrm{o}}$, obtained by completing the corner blocks using the independent copies of the off-diagonal blocks of $H$, and $H_{I^{\txt{c}}}$ is defined similarly. Clearly, the entry $(H_{W,N}-\wh H_{W,N})_{ab}$ vanishes, unless $a,b\in\qa{x_1}\cup\qa{x_2}\cup\qa{y_1}\cup\qa{y_2}$. For any $z\in \C$, denote $\wh G(z):=(\wh H-z)^{-1}$ and regard $G_{W,R}(z)$ as $(H_{I}-z)^{-1}$. Then, the identity \eqref{eq_xx_to_1_N_Tr} shows
        \begin{equation}
            \frac{1}{R}\E\Tr\, G_{W,R}(E+\ii \eta)=\E \qb{(H_{I}-E-\ii\eta)^{-1}}_{\sx\sx}=\E \qb{(\wh H-E-\ii\eta)^{-1}}_{\sx\sx}=\E \wh G_{W,N}(E+\ii \eta)_{\sx\sx}.
        \end{equation}
        Therefore, for any $s\in(0,1/2)$, we have
\begin{align}
                &\absa{\frac{1}{N}\E\Tr\, G_{W,N}(E+\ii \eta)-\frac{1}{R}\E\Tr\, G_{W,R}(E+\ii \eta)}=\absa{\E G_{W,N}(E+\ii \eta)_{\sx\sx}-\E \wh G_{W,N}(E+\ii \eta)_{\sx\sx}}\nonumber\\
                &\qquad\qquad\qquad\qquad\leq \sum_{a,b\in\qa{x_1}\cup\qa{x_2}\cup\qa{y_1}\cup\qa{y_2}}\E\abs{G_{W,N}(E+\ii \eta)_{\sx a}}(|H_{ab}|+|\wh H_{ab}|)\abs{\wh G_{W,N}(E+\ii \eta)_{b\sx}}\label{eq_coupling_difference}\\
                &\qquad\qquad\qquad\qquad\leq \eta^{s-2} \sum_{a,b\in\qa{x_1}\cup\qa{x_2}\cup\qa{y_1}\cup\qa{y_2}}\E\abs{G_{W,N}(E+\ii \eta)_{\sx a}}^{s}(|H_{ab}|+|\wh H_{ab}|)\nonumber\\
                &\qquad\qquad\qquad\qquad\leq 2 \eta^{-2} \sum_{a,b\in\qa{x_1}\cup\qa{x_2}\cup\qa{y_1}\cup\qa{y_2}}\pb{\E\abs{G_{W,N}(E+\ii \eta)_{\sx a}}^{2s}}^{1/2}\pb{\E|H_{ab}|^2+\E|\wh H_{ab}|^2}^{1/2}.\nonumber
            \end{align}
Taking $s=q/2$ with $q$ given in \eqref{eq_fractional_decay_bound_H_E_i_eta}, using the bound \eqref{eq_dist_sx_boundary}, and the fact $\E|H_{ab}|^2\leq W^{-1}$, we obtain
        \begin{equation}
            \begin{aligned}
                &\absa{\frac{1}{N}\E\Tr\, G_{W,N}(E+\ii \eta)-\frac{1}{R}\E\Tr\, G_{W,R}(E+\ii \eta)}\\
            \leq& CW^{-1/2}\eta^{-2}\sum_{a,b\in\qa{x_1}\cup\qa{x_2}\cup\qa{y_1}\cup\qa{y_2}}\pb{\E\abs{G_{W,N}(E+\ii \eta)_{\sx a}}^{2s}}^{1/2}\leq CW^{C}\eta^{-2}\exp(-cN/W^2).
            \end{aligned}
        \end{equation}
        This proves the first estimate \eqref{eq_Tr_G_N_M_comparison} and completes the proof of \Cref{lemma_N_N_comparison}.
    \end{proof}

    Next, we prove \Cref{lemma_rho_M_1_M_0_comparison} by explaining how to modify the proof of \Cref{lemma_N_N_comparison}. The key idea is that, in the analogue of \eqref{eq_coupling_difference}, the local law in \Cref{lemma_local_law} allows us to obtain much sharper spatial decay than the fractional moment estimate \eqref{eq_fractional_decay_bound_H_E_i_eta}.

    \begin{proof}
        \Cref{lemma_rho_M_1_M_0_comparison} follows from a coupling argument analogous to that in the proof of \Cref{lemma_N_N_comparison}. Let $M_0$ and $M_1$ take the place of $N$ and $N'$ in the proof of \Cref{lemma_N_N_comparison}. In particular, we have $R=2\ceil{M_0/4W}W$. Then, the bound \eqref{eq_dist_sx_boundary} becomes
        \begin{equation}\label{eq_dist_sx_boundary_M_0}
            \mathrm{dist}\pa{\sx,\qa{x_1}\cup\qa{x_2}\cup\qa{y_1}\cup\qa{y_2}}\geq M_0/100,
        \end{equation}
        and we can prove the following analogue of (the second step of) \eqref{eq_coupling_difference} for each $i=0,1$:
\begin{align}
                &\absa{\frac{1}{M_i}\E\Tr\, G_{W,M_i}(E+\ii \eta)-\frac{1}{R}\E\Tr\, G_{W,R}(E+\ii \eta)}\nonumber\\
                \leq& \sum_{a,b\in\mathsf{B}(I)}\E\abs{G_{W,M_i}(E+\ii \eta)_{\sx a}}(|H_{ab}|+|\wh H_{ab}|)\abs{\wh G_{W,M_i}(E+\ii \eta)_{b\sx}}\label{eq_coupling_difference_M_0}\\
                \leq &\eta^{-1} \sum_{a\in \mathsf{B}(I)}\sum_{b\in \pa{\mathsf{B}(I)}\cap I}\E(|H_{ab}|+|\wh H_{ab}|)\abs{G_{W,R}(E+\ii \eta)_{b\sx}}\leq C W^{C} \sum_{a\in\mathsf{B}(I)}\sum_{b\in\mathsf{B}(I)\cap I}\pb{\E\abs{G_{W,R}(E+\ii \eta)_{b\sx}}^{2}}^{1/2}.\nonumber
            \end{align}
Here, we denote $\mathsf{B}(I):=\qa{x_1}\cup\qa{x_2}\cup\qa{y_1}\cup\qa{y_2}$, and, in the second step, we used the fact that $\wh G_{W,M_i}(E+\ii \eta)_{b\sx}=0$ for every $b\notin I$ and $\wh G_{W,M_i}(E+\ii \eta)_{b\sx}=G_{W,R}(E+\ii\eta)_{b\sx}$ for every $b\in I$. We also recall the fact $\eta=W^{-7/4}$ here. Fixing any large constant $D'\geq 1$, for sufficiently large $W$, by \eqref{eq_dist_sx_boundary_M_0} and the local law \eqref{eq_local_law}, with probability at least $1-W^{-D'}$, we can bound $\absa{G_{W,R}(E+\ii\eta)_{b\sx}}$ as
        \begin{equation}
            \sup_{b\in \mathsf{B}(I)\cap I}\absa{G_{W,R}(E+\ii\eta)_{b\sx}}\leq\pa{\frac{M_0}{\ell_{\eta}}+1}^{-D_0}\leq CW^{-D_0/16},
        \end{equation}
        where $D_0\geq 1$ may be chosen arbitrarily large. Finally, given any fixed large constant $D\geq 1$, taking $D_0$ to be large enough, substituting this back into \eqref{eq_coupling_difference_M_0}, and using the trivial bound $\|G_{W,M_i}(E+\ii \eta)\|\vee\|G_{W,R}(E+\ii \eta)\|\leq \eta^{-1}$, we obtain
        \begin{equation}
            \absa{\frac{1}{M_i}\E\Tr\, G_{W,M_i}(E+\ii \eta)-\frac{1}{R}\E\Tr\, G_{W,R}(E+\ii \eta)}\leq CW^{-D}.
        \end{equation}
        Since all these estimates are uniform in $M_1\geq 2M_0$ and $E\in K$, this completes the proof of \Cref{lemma_rho_M_1_M_0_comparison} by taking the supremum over $M_1\geq 2M_0$ and $E\in K$.
    \end{proof}

    \section{Proof of \texorpdfstring{\Cref{theorem_Poisson_process}}{Theorem 1.2}}\label{sec_proof_of_main_result}
    This section is devoted to the proof of \Cref{theorem_Poisson_process} for the periodic model $H\equiv H_{W,N}$. With all these preparations, the argument here is essentially a direct extension of the following classical idea from the literature on random Schr{\" o}dinger operators (see, for example, \cite{Minami1996Localfluctuation,GerminetKloop2014Spectralstatistics}): \emph{localization, together with appropriate control of the density of states, leads to Poisson statistics.}
To summarize the basic ideas, heuristically, we first prove the asymptotic Poisson statistics for the random matrix obtained from $H$ by partitioning $\ZN$ into smaller boxes and removing the  hopping blocks between them, and then show the asymptotic Poisson statistics by a comparison between the new matrix and $H$. For example, in the extreme case, if we partition $\ZN$ as
    \begin{equation}
        \ZN:=\bigsqcup_{\ell=0}^{n-1}\qa{\ell W},
    \end{equation}
    the hopping blocks of $H$ will be all removed, and the matrix then becomes simply a block diagonal matrix, in which the asymptotic Poisson statistics can be easily proved. Therefore, in the first step, we want the partition to be sufficiently fine. On the other hand, clearly, as a trade-off, finer partitions require more elimination of hopping blocks, which increases the difficulty of the comparison step. Therefore, the key to the proof is to handle the balance between these two issues.

    \begin{proof}
        We first prove the convergence \eqref{eq_xi_to_PPP_d_lambda}. Throughout this proof, we fix an energy $E\in(-2,2)$, abbreviate $\rho\equiv \rho_W(E)$, and set
        \begin{equation}
            L\equiv L_{W,N}:=W\sqrt{N\log N}.
        \end{equation}
        By the assumption \eqref{eq_localization_assumption}, this intermediate scale parameter satisfies the following two crucial properties, which efficiently allow us to obtain the asymptotic Poisson statistics:
        \begin{equation}
            \lim_{N\to\infty}L/N=0,\qquad \lim_{N\to \infty} L/(W^2\log N)=\infty.
        \end{equation}
        To obtain the Poisson statistics, we partition the discrete circle $\ZN$ into $r\equiv r_{W,N}$ successive boxes $\ha{I_{k}}_{k=1}^{r}$, each of which consists of some successive blocks. Moreover, this partition satisfies the following conditions:
        \begin{equation}
            L\leq N_{k}\leq 2L,\qquad \sum_{k=1}^{r}N_{k}=N,
        \end{equation}
        where $N_k:=|I_{k}|$ is the number of vertices in the $k$-th box. Consequently, the number $r$ of boxes satisfies
        \begin{equation}
            N/(2L)\leq r\leq N/L,
        \end{equation}
        which, in particular, implies $r\to \infty$. 
        Next, suppose that, for any $k\in\qq{r}$, $H_{I_k}^{\mathrm{o}}$ is the restriction of $H$ on the box $I_k$, and $H_{I_k}$ is the corresponding periodic form obtained by completing the corner blocks using independent copies of the off-diagonal blocks of $H$. Also, we define
        \begin{equation}
            \wh H\equiv\wh H_{W,N}:=\bigoplus_{k=1}^{r}H_{I_k},\qquad \wh G(z)\equiv \wh G_{W,N}(z):=(\wh H-z)^{-1}.
        \end{equation}
        Moreover, for any $k\in\qq{r}$, we define
        \begin{equation}
            \xi_{E,k}\equiv \xi_{E,k,W,N}:=\sum_{i=1}^{N_k}\delta_{N\rho\p{\lambda_i^{(k)}-E}},
        \end{equation}
        where $\lambda_1^{(k)}\leq \cdots \leq \lambda_{N_{k}}^{(k)}$ are the eigenvalues of $H_{I_k}$. Then, clearly, the local statistics $\wh \xi_{E}\equiv \wh \xi_{E,W,N}$ associated with $\wh H$ satisfy
        \begin{equation}\label{eq_wh_xi_E_wh_xi_E_k}
            \wh \xi_{E}=\sum_{k=1}^{r}\xi_{E,k}.
        \end{equation}
        We claim that the convergence \eqref{eq_xi_to_PPP_d_lambda} holds with $\xi_E$ replaced by $\wh \xi_{E}$, i.e., for every $f\in C_{c}(\R)$ with $f\geq 0$, we have
        \begin{equation}\label{eq_wh_xi_to_PPP_d_lambda}
            \lim_{N\to\infty} \E\exp\pa{-\int_{\R}f\, \rd \wh\xi_{E}}=\exp\pa{-\int_{\R}\qa{1-\exp\pa{-f(\lambda)}}\, \rd \lambda}.
        \end{equation}
        By the independence of $\ha{H_{I_k}}_{k=1}^{r}$ and the identity \eqref{eq_wh_xi_E_wh_xi_E_k}, the expectation on the left-hand side of \eqref{eq_wh_xi_to_PPP_d_lambda} can be written as
        \begin{equation}\label{eq_E_f_wh_xi_independent_product}
            \E\exp\pa{-\int_{\R}f\, \rd \wh\xi_{E}}=\prod_{k=1}^{r}\E\exp\pa{-\int_{\R}f\, \rd \xi_{E,k}}.
        \end{equation}
        Suppose that $f\in C_{c}(\R)$ with $f\geq 0$ is supported in a finite interval $J\subseteq \R$. 
        For each $k\in\qq{r}$, we claim the bound
        \begin{equation}\label{eq_inclusion_exclusion_bound}
            1-\int_{\R}[1-\exp(-f)]\, \rd \xi_{E,k}\leq \exp\pa{-\int_{\R}f\, \rd \xi_{E,k}}\leq \frac{Q_k(Q_k-1)}{2}+ 1-\int_{\R}[1-\exp(-f)]\, \rd \xi_{E,k},
        \end{equation}
        where $Q_k:=\int_{J}\rd \xi_{E,k}=N_{E+J/N\rho}(H_{I_k})$ denotes the number of eigenvalues of $H_{I_k}$ in the interval $E+J/N\rho$. 
        Specifically, denote
        \begin{equation}
            q_i:=1-\exp[-f(N\rho(\lambda_i^{(k)}-E))]\in[0,1],\qquad \forall i\in\qq{N_k}.
        \end{equation}
        Then, by the inclusion-exclusion bounds, we have
        \begin{equation}
            \begin{aligned}
                &0\leq \int_{\R}[1-\exp(-f)]\, \rd \xi_{E,k}+\exp\pa{-\int_{\R}f\, \rd \xi_{E,k}}-1\\
                &\qquad\qquad\qquad=\sum_{i=1}^{N_k}q_i-1+\prod_{i=1}^{N_k}(1-q_i)\leq \sum_{1\leq i<j\leq N_k}q_iq_j\leq \frac{Q_k(Q_k-1)}{2}.
            \end{aligned}
        \end{equation}
        Taking expectations on both sides of \eqref{eq_inclusion_exclusion_bound} and substituting it back into \eqref{eq_E_f_wh_xi_independent_product}, we obtain
        \begin{equation}\label{eq_E_f_wh_xi_E_Q_k_product}
            \E\exp\pa{-\int_{\R}f\, \rd \wh\xi_{E}}=\prod_{k=1}^{r}\pa{R_{k}+ 1-\E\int_{\R}[1-\exp(-f)]\, \rd \xi_{E,k}},
        \end{equation}
        where $R_k$ is an error term satisfying $R_k\in[0,\E Q_k(Q_k-1)/2]$ for every $k\in\qq{r}$. 
        By the Minami estimate in \eqref{eq_Wegner_Minami_estimate}, we have
        \begin{equation}
            \sum_{k=1}^{r}\frac{1}{2}\E Q_k(Q_k-1)\leq \frac{1}{(N\rho)^2}\sum_{k=1}^{r}(CN_{k}|J|)^2\leq \pa{\frac{C|J|}{N\rho}}^2\cdot 2L\cdot\sum_{k=1}^r N_{k}\leq 2\pa{\frac{C|J|}{\rho}}^2\frac{L}{N}\to 0.
        \end{equation}
        By the Wegner estimate in \eqref{eq_Wegner_Minami_estimate}, we have
        \begin{equation}
            \begin{aligned}
                0\leq \E\int_{\R}[1-\exp(-f)]\, \rd \xi_{E,k}\leq \E Q_k\leq \frac{2C|J|L}{N\rho}\to 0.
            \end{aligned}
        \end{equation}
        Moreover, by the identity \eqref{eq_wh_xi_E_wh_xi_E_k}, for any bounded measurable function $h$ with compact support, we have
        \begin{equation}
            \begin{aligned}
                &\sum_{k=1}^{r}\E\int_{\R}h\, \rd \xi_{E,k}=\E\int_{\R}h\, \rd \wh \xi_{E}=\int_{\R}h(\lambda)\sum_{k=1}^{r}\frac{N_k}{N}\frac{\rho_{W,N_k}(E+\lambda/N\rho)}{\rho_W(E)}\, \rd \lambda=\int_{\R}h(\lambda)\, \rd \lambda\\
                &+\int_{\R}h(\lambda)\sum_{k=1}^{r}\frac{N_k}{N}\frac{\rho_{W,N_k}(E+\lambda/N\rho)-\rho_{W}(E+\lambda/N\rho)}{\rho_W(E)}\, \rd \lambda+\int_{\R}h(\lambda)\frac{\rho_{W}(E+\lambda/N\rho)-\rho_{W}(E)}{\rho_W(E)}\, \rd \lambda.
            \end{aligned}
        \end{equation}
        Consider the second and the third terms on the right-hand side, and recall from \eqref{eq_rho_W_uniform_bounds} that $\rho=\rho_{W}(E)$ is uniformly bounded away from zero. Applying \eqref{eq_rho_W_N_rho_W} and the fact that $N_k/(W^2\log N)\to \infty$ for the second term, together with the equicontinuity of $\ha{\rho_W}_{W=1}^{\infty}$ from \Cref{lemma_rho_W_N} for the third term, we can show that both of these two terms tend to $0$. In sum, we obtain
        \begin{equation}
            \lim_{N\to \infty}\sum_{k=1}^{r}\E\int_{\R}h\, \rd \xi_{E,k}=\int_{\R}h(\lambda)\, \rd \lambda.
        \end{equation}
        Therefore, denoting $R_k':=\E\int_{\R}[1-\exp(-f)]\, \rd \xi_{E,k}$, we have
        \begin{equation}
            \lim_{N\to \infty}\max_{k=1}^{r}|R_k|\vee |R_k'|=0,\quad \lim_{N\to \infty}\sum_{k=1}^{r}R_k=0,\quad\lim_{N\to \infty}\sum_{k=1}^{r}R_k'=\int_{\R}\qa{1-\exp(-f(\lambda))}\, \rd \lambda.
        \end{equation}
        Substituting this back into \eqref{eq_E_f_wh_xi_E_Q_k_product}, by direct calculus, we have
        \begin{equation}
            \begin{aligned}
                \E\exp\pa{-\int_{\R}f\, \rd \wh\xi_{E}}=\prod_{k=1}^{r}\pa{1+R_k-R_k'}\to \exp\pa{-\int_{\R}[1-\exp(-f(\lambda))]\, \rd \lambda}.
            \end{aligned}
        \end{equation}
        This yields the convergence \eqref{eq_wh_xi_to_PPP_d_lambda}.

        Next, we prove the convergence \eqref{eq_xi_to_PPP_d_lambda} by comparing $H$ and $\wh H$. Denote
        \begin{equation}\label{eq_Delta_H_wh_H}
            \Delta:=H-\wh H.
        \end{equation}
        Clearly, for any $f\in C_{c}(\R)$ with $f\geq 0$, we have
        \begin{equation}
            \absa{\E\exp\pa{-\int_{\R}f\, \rd \wh\xi_{E}}-\E\exp\pa{-\int_{\R}f\, \rd \xi_{E}}}\leq \E\absa{\int_{\R}f\, \rd \xi_{E}-\int_{\R}f\, \rd \wh\xi_{E}}.
        \end{equation}
        Therefore, by \eqref{eq_wh_xi_to_PPP_d_lambda}, it suffices to prove, for any $f\in C_{c}(\R)$,
        \begin{equation}\label{eq_E_f_xi_E_f_wh_xi_E}
            \lim_{N\to \infty}\E\absa{\int_{\R}f\, \rd \xi_{E}-\int_{\R}f\, \rd \wh\xi_{E}}=0.
        \end{equation}
        Moreover, for any $f\in L^{1}(\R)$, the definition of $\xi_{E}$ yields
        \begin{align*}
            \absa{\E\int_{\R}f\, \rd \xi_{E}}\leq \E\int_{\R}|f|\, \rd \xi_{E}=\int_{\R}|f(\lambda)|\frac{\rho_{W,N}(E+\lambda/N\rho)}{\rho}\, \rd \lambda\leq C_{E}\|f\|_{1},
        \end{align*}
        and we have $\absa{\E\int_{\R}f\, \rd \wh\xi_{E}}\leq C_{E}\|f\|_{1}$ by a similar argument. Here, we used the bounds \eqref{eq_regularity_rho_W_N} and \eqref{eq_rho_W_uniform_bounds} in the last step. 
        Then, by the Poisson $L^{1}$-approximation (see, e.g., \cite[Theorem 9.5]{rudin1991functional}), the problem can be further reduced to proving \eqref{eq_E_f_xi_E_f_wh_xi_E} for the following family of functions:
        \begin{equation}
            \ha{f_{\zeta}:\lambda\mapsto\mathrm{Im}\, \frac{1}{\lambda-\zeta}:\zeta\in\C_+}.
        \end{equation}
        Fix $\zeta\in\C_+$ and denote $\eta\equiv\eta_{N}:=\mathrm{Im}\, \zeta/N\rho$ for the remainder of the proof of \eqref{eq_xi_to_PPP_d_lambda}. Then, the spectral theorem yields
        \begin{equation}
            \int_{\R}f_{\zeta}\, \rd \xi_{E}=\frac{1}{N\rho}\mathrm{Im}\,\Tr\,G\pa{E+\frac{\zeta}{N\rho}},\qquad \int_{\R}f_{\zeta}\, \rd \wh \xi_{E}=\frac{1}{N\rho}\mathrm{Im}\,\Tr\, \wh G\pa{E+\frac{\zeta}{N\rho}}.
        \end{equation}
        It therefore suffices to show
        \begin{equation}
            \sum_{x\in\ZN}\E\absa{\mathrm{Im}\,G\pa{E+\frac{\zeta}{N\rho}}_{xx}-\mathrm{Im}\,\wh G\pa{E+\frac{\zeta}{N\rho}}_{xx}}=\oo\pa{N}.
        \end{equation}
        To this end, for every $k\in\qq{r}$, let $\mathsf{B}(I_k)\subseteq \ZN$ be the set of indices contained in the boundary blocks of $I_k$, and denote
        \begin{equation}
            \mathsf{B}:=\bigcup_{k=1}^{r}\mathsf{B}(I_k),\qquad \mathsf{B}_{\ell}:=\ha{x\in\ZN:\mathrm{dist}(x,\mathsf{B})<\ell},\qquad \forall \ell\in \Z_+.
        \end{equation}
        For the remainder of the proof of \eqref{eq_xi_to_PPP_d_lambda}, we also set
        \begin{equation}
            \ell\equiv\ell_{W,N}:=W\ceil{AW\log N},
        \end{equation}
        where $A\geq 1$ is a fixed large constant, to be chosen below. Clearly, we have, for some large constant $C\geq 1$,
        \begin{equation}
            |\mathsf{B}_{\ell}|\leq C r\ell,
        \end{equation}
        and the assumption \eqref{eq_localization_assumption} provides
        \begin{equation}
            \lim_{N\to\infty} \ell/L=0.
        \end{equation}
        Then, by the bound \eqref{eq_E_Im_G_xx_bound}, we have, for some large constant $C\geq 1$,
        \begin{equation}
            \sum_{x\in\mathsf{B}_{\ell}}\E\absa{\mathrm{Im}\,G\pa{E+\frac{\zeta}{N\rho}}_{xx}-\mathrm{Im}\,\wh G\pa{E+\frac{\zeta}{N\rho}}_{xx}}\leq CN\cdot \frac{r\ell}{N}\leq CN\cdot\frac{\ell}{L}=\oo(N).
        \end{equation}
        It remains to bound
        \begin{equation}
            \sum_{x\notin\mathsf{B}_{\ell}}\E\absa{\mathrm{Im}\,G\pa{E+\frac{\zeta}{N\rho}}_{xx}-\mathrm{Im}\,\wh G\pa{E+\frac{\zeta}{N\rho}}_{xx}}.
        \end{equation}
        Define the deterministic support of $\Delta$ (recall \eqref{eq_Delta_H_wh_H}) as
        \begin{equation}
            \cD\equiv \cD_{W,N}:=\ha{(a,b)\in\ZN^2:\E|\Delta_{ab}|^2\neq 0}.
        \end{equation}
        Since $\ell=\oo\pa{L}$, for some large constant $C\geq 1$, the set $\cD$ satisfies the following three properties:
        \begin{equation}\label{eq_cD_properties}
            |\cD|\leq CrW^2,\qquad \sup_{(a,b)\in\cD}\E|\Delta_{ab}|^2\leq CW^{-1},\qquad \cD\subseteq \mathsf{B}\times\mathsf{B}.
        \end{equation}
        Then, for any $x\notin\mathsf{B}_{\ell}$ and $s\in(0,1/2)$, applying the operator norm bound $\|\wh G\|\leq \eta^{-1}$ and H{\" o}lder's inequality, we obtain
        \begin{equation}
            \begin{aligned}
                &\E\absa{G\pa{E+\frac{\zeta}{N\rho}}_{xx}-\wh G\pa{E+\frac{\zeta}{N\rho}}_{xx}}\leq \sum_{(a,b)\in\cD}\E\absa{G\pa{E+\frac{\zeta}{N\rho}}_{xa}}\cdot|\Delta_{ab}|\cdot\absa{\wh G\pa{E+\frac{\zeta}{N\rho}}_{bx}}\\
                &\leq \eta^{s-2}\sum_{(a,b)\in\cD}\E\absa{G\pa{E+\frac{\zeta}{N\rho}}_{xa}}^{s}\cdot|\Delta_{ab}|\leq \eta^{-2}\sum_{(a,b)\in\cD}\qa{\E\absa{G\pa{E+\frac{\zeta}{N\rho}}_{xa}}^{2s}}^{1/2}\pa{\E|\Delta_{ab}|^2}^{1/2}.
            \end{aligned}
        \end{equation}
        Taking $s=q/2$, where $q$ is the constant from the fractional moment bound \eqref{eq_fractional_decay_bound_H_E_i_eta}, and applying the bounds in \eqref{eq_cD_properties}, we can bound $G\pa{E+\zeta/N\rho}_{xx}-\wh G\pa{E+\zeta/N\rho}_{xx}$ as
        \begin{equation}
            \E\absa{G\pa{E+\frac{\zeta}{N\rho}}_{xx}-\wh G\pa{E+\frac{\zeta}{N\rho}}_{xx}}\leq C rW^{C}\eta^{-2}\exp\pa{-c\ell/W^2}\leq C'N^{C'}\exp(-c'\ell/W^2).
        \end{equation}
        Here, $C,C'\geq 1$ and $c,c'\in(0,1)$ are $A$-independent constants. Taking $A$ to be large enough, e.g., $c'A\geq 100C'$, yields, uniformly in $x\notin \mathsf{B}_{\ell}$,
        \begin{equation}
            \E\absa{G\pa{E+\frac{\zeta}{N\rho}}_{xx}-\wh G\pa{E+\frac{\zeta}{N\rho}}_{xx}}=\oo\pa{1}.
        \end{equation}
        Therefore, we obtain
        \begin{equation}
            \sum_{x\notin\mathsf{B}_{\ell}}\E\absa{\mathrm{Im}\,G\pa{E+\frac{\zeta}{N\rho}}_{xx}-\mathrm{Im}\,\wh G\pa{E+\frac{\zeta}{N\rho}}_{xx}}=\oo\pa{N}.
        \end{equation}
        This proves the convergence \eqref{eq_xi_to_PPP_d_lambda}. Then, the convergence \eqref{eq_convergence_to_rho_sc_E_d_lambda} follows from a direct rescaling and \eqref{eq_rho_W_to_rho_sc}.

        Finally, we prove the convergence of correlation functions. Given any finite interval $J$, by the Minami estimate in \eqref{eq_Wegner_Minami_estimate}, for some large constant $C\geq 1$, we have uniformly in $N\geq 1$ and $k\geq 1$,
        \begin{equation}
            \E(\xi_{E,W,N}(J))_{k}\leq (C_{E}|J|)^{k}.
        \end{equation}
        Here, we employ the notation $(\cdot)_{l}$, defined for any $k\geq 0$ and $l\geq 1$ by
        \begin{equation}
            (k)_{l}:=
            \begin{cases}
                k(k-1)\cdots (k-l+1), & l\leq k,\\
                0, & l>k.
            \end{cases}
        \end{equation}
        This immediately implies, for any fixed integer $p\geq 1$,
        \begin{equation}
            \sup_{W,N}\E|\xi_{E,W,N}(J)|^p<\infty.
        \end{equation}
        These bounds verify the moment hypotheses of \cite[Theorem 2.2]{zessin1983method}, and therefore establish the convergence of correlation functions. This completes the proof of \Cref{theorem_Poisson_process}.
    \end{proof}

     \section{Open Problems}\label{sec_open_problems}
     In this section, we discuss some related open problems.
     \begin{enumerate}
         \item {\bf Universality of the Poisson statistics for random band matrices.} The techniques in this work largely rely on the Gaussian distribution of the entries. In contrast, the localization of eigenvectors in \cite{Localization1_2}, the delocalization of eigenvectors and the asymptotic RMT statistics of eigenvalues in \cite{erdos2025zigzagstrategyrandomband} apply to broader classes of entry distributions. Therefore, the question is, how far can the present Poisson statistics result be extended beyond the Gaussian setting? In particular, what local statistics arise for singular entry distributions, such as Bernoulli distributions, in the localized regime?

         \item {\bf Beyond compactly supported variance profiles.} The method in this work and \cite{Localization1_2} can be extended to random band matrices with general compactly supported variance profiles, including the sharp-cutoff profile. On the other hand, compared to the results in the delocalized phase for general rapidly decaying profiles \cite{erdos2025zigzagstrategyrandomband} and power-law profiles in \cite{fan2026localizationlengthspowerlawrandom}, existing localization results apply to a more restricted class of profiles. The extension of the current methods to these profiles seems to face several difficulties.

         \item {\bf Sharp transition of the local statistics and the critical statistics.} \Cref{theorem_Poisson_process} in this work proves the asymptotic Poisson statistics for $W\ll \sqrt{N/\log N}$, and the results from \cite{Band1D,erdos2025zigzagstrategyrandomband} establish the asymptotic RMT statistics for $W\geq N^{1/2+c}$, where $c>0$ is a fixed arbitrarily small constant. However, the physical conjecture predicts a sharp Poisson-RMT transition from $W\ll \sqrt{N}$ to $W\gg \sqrt{N}$. Furthermore, a similar sharp transition has been rigorously studied at the edge in \cite{liu2025edgestatisticsrandomband}. Therefore, it is natural to determine the asymptotic behavior of the local statistics in the remaining range $\sqrt{N/\log N}\ll W\leq N^{1/2+c}$, including the critical regime $W\asymp \sqrt{N}$.

         \item {\bf Dynamical localization and long-time quantum evolution.} Suppose $W^2\ll N$ and consider the quantum transition probability, defined as
         \begin{equation}
             P_{x}(t)\equiv P_{x}(t,W,N):=\E|\exp(\ii t H_{W,N})_{0x}|^2,\qquad \forall x\in \ZN,~ t\in[0,\infty).
         \end{equation}
         By the $t\sim \eta^{-1}$ correspondence of the quantum evolution, the regime $0\leq t\lesssim W^2$ lies in the range of applicability of local law. For $t\gg W^2$, \cite[Theorem 1]{Localization1_2} implies a dynamical localization of this transition probability. Then, it is natural to study the asymptotic behavior of the transition probability $\ha{P_{x}(t)}_{x\in\ZN}$ for $t\gg W^2$, and try to establish an analogue of \cite[Theorem 3.1]{erdHos2011quantum1}.

         \item {\bf Localization at the spectral edge.} At the spectral edge, local statistics are well-studied in \cite{liu2025edgestatisticsrandomband,liu2025edgeuniversalityinhomogeneousrandom,liu2026edgeuniversalityinhomogeneousrandom}. However, for eigenvectors, \cite{Localization1_2} only provides the exponential decay bound on the scale $W^2$, which is the correct conjectured localization length of the bulk eigenvectors. At the spectral edge, the conjectured localization length becomes $W^{6/5}$. Achieving this improvement seems to require new ideas and techniques.
     \end{enumerate}

\end{document}